\documentclass[12pt,reqno]{amsart}

\usepackage[margin=1in]{geometry}

\title{On the loss of continuity for super-critical drift-diffusion equations}

\date{\today}

\author{Luis Silvestre}
\address{Department of Mathematics, The University of Chicago}
\email{luis@math.uchicago.edu}

\author{Vlad Vicol}
\address{Department of Mathematics, Princeton University}
\email{vvicol@math.princeton.edu}

\author{Andrej Zlato\v{s}}
\address{Department of Mathematics, University of Wisconsin--Madison}
\email{andrej@math.wisc.edu}

\usepackage{tikz}
\usepackage{times}

\usepackage{hyperref}
\usepackage{color}

\theoremstyle{plain}
\newtheorem{theorem}{Theorem}[section]

\newtheorem{lemma}[theorem]{Lemma}
\newtheorem{proposition}[theorem]{Proposition}
\newtheorem{corollary}[theorem]{Corollary}

\theoremstyle{definition}
\newtheorem{remark}[theorem]{Remark}

\def\tilde{\widetilde}
\numberwithin{equation}{section}

\renewcommand\hat{\widehat}

\def\ZZ{{\mathbb Z}}
\def\RR{{\mathbb R}}

\def\DD{\mathcal D}
\def\CC{\mathcal C}

\def\EL{\mathcal L}

\newcommand{\eps}{\varepsilon}

\newcommand{\grad} {\nabla}
\newcommand{\lap}{\Delta}

\newcommand{\dd}{\, \mathrm{d}}

\newcommand{\sign}{\mathrm{sgn}}

\DeclareMathOperator*{\osc}{osc}

\DeclareMathOperator*{\sgn}{sgn}

\def\grad{{\nabla}}

\begin{document}


\begin{abstract}
We show that there exist solutions of drift-diffusion equations in two dimensions with divergence-free super-critical drifts, that become discontinuous in finite time.
We consider classical as well as fractional diffusion. 
However, in the case of classical diffusion and time-independent drifts 
we prove that solutions satisfy a modulus of continuity depending only on the local $L^1$ norm of the drift, which is a super-critical quantity.
\end{abstract}


\maketitle

\section{Introduction} \label{sec:intro}
We study the continuity of solutions to equations with divergence-free drift and fractional or classical diffusion. We prove that in supercritical regimes, there are solutions which become discontinuous in finite time. However, we also prove that in two dimensions, a solution to a drift diffusion equation  with classical diffusion stays continuous if the drift is only locally bounded in $L^1$ (which has supercritical scaling), provided the drift is time-independent.

Equations with drift and diffusion appear in numerous places in mathematical physics. In many cases, the drift depends on the solution and the equation is nonlinear. A successful understanding of well-posedness of the problem in each case depends on the a priori estimates that can be established. In most cases, these are based on the linearized drift-diffusion equation, which provides the motivation for this work.

We consider the problem of continuity of solutions to the Cauchy problem
\begin{align}
& \partial_t \theta + u \cdot \grad \theta + (-\Delta)^s \theta= 0 \label{evo}\\
& \theta(0,\cdot) = \theta_0 \label{evo:IC}
\end{align}
where $s\in(0,1]$, and $u$ is a given divergence-free vector field.

For each $s>0$, the equation has a natural scaling: if $\theta(t,x)$ is a solution of \eqref{evo} with drift $u(t,x)$, then $\theta_\lambda(t,x) = \lambda^{2s-1}\theta(\lambda^{2s} t, \lambda x)$ is a also a solution of \eqref{evo}, but with drift given by $u_\lambda(t,x)= \lambda^{2s-1} u(\lambda^{2s}t, \lambda x)$. A Banach space $X$, with norm $\| \cdot \|_{X}$,
is called {\em critical} with respect to the natural scaling, if $\| u_\lambda\|_X = \| u\|_X$ for all $\lambda > 0$. If on the other hand, $\| u_\lambda\|_{X} \to \infty$ as $\lambda \to 0$, the space is called {\em supercritical}.
In the supercritical cases, when one zooms in at a point (i.e., sends $\lambda \to 0$), the bound on the drift becomes worse, so that regularity of the solutions cannot be inferred from linear perturbation theory. In view of the scaling described above, for $s \in (0,1/2)$ a critical space for \eqref{evo} is the H\"older space $\dot{C}^{1-2s}$, for $s=1/2$ it is the Lebesgue space $L^\infty$, while for $s \in (1/2,1]$ it is $L^{d/(2s - 1)}$. 

In the context of fluid mechanics, the case of divergence-free drifts is of special importance due to incompressibility, while (fractional) diffusion is a regularizing term, for instance, in the well known surface quasi-geostrophic (SQG) model \cite{ConstantinMajdaTabak94,ConstantinWu99}. The possibility of finite time blowup for the SQG equation with supercritical fractional diffusion is an outstanding open problem. One could speculate that the divergence-free character of the drift plays an important role in the well-posedness of the problem. In fact, blow up in finite time does not seem to be known to hold not just for SQG, but for most (if not all) of the supercritical active scalar equations with divergence-free drift currently in the literature (see, for example \cite{CCGR09,ChaeConstantinCordobaGancedoWu12,ConstantinIyerWu2008,FriedlanderRusinVicol11}). In contrast, for some drift-diffusion equations with non-divergence-free drifts, a blow up scenario is possible, and well understood. For example, this is the case in Burgers equation with fractional diffusion \cite{AlibaudDroniouJulien07,KiselevNazarovShterenberg08}, the Keller-Segel model \cite{DolbeaultPerthame2004}, and many other equations \cite{ConstantinLaxMajda85,CordobaCordobaFontelos06, LiRodrigo09}. Even self-similar blow up may  be sometimes obtained, since there is no mechanism which prevents mass-concentration.

Indeed, the divergence-free condition on the drift is known to imply some qualitative properties of the solution. For pure transport equations without diffusion, the flow is well defined almost everywhere just assuming that the drift is in the Sobolev space $W^{1,1}$ instead of the classical Lipschitz assumption of Picard's theorem (see \cite{Ambrosio05} and \cite{DiPernaLions1989}). Also, certain type of singularities are ruled out for divergence-free drifts, as in \cite{CordobaFefferman01}. We give another example of this phenomenon in Section~\ref{sec:trajectories}, where we prove that, in two dimensions, non-vanishing continuous divergence-free vector fields (not necessarily Lipschitz or even H\"older) have unique trajectories {(see also~\cite{BochutDesvillettes01})}. For equations with drifts and diffusion, the divergence-free condition has been used to obtain some estimates which are independent of the size of the drift, e.g. first eigenvalue estimates \cite{BerestyckiHamelNadirashvili05}, mixing rates \cite{ConstantinKiselevRyzhikZlatos08}, and expected exit times \cite{IyerNovikovRyzhikZlatos10}.

For scalar equations with drift and classical diffusion ($s=1$ in \eqref{evo}), if the drift $u$ is assumed to be divergence-free and  in the critical space $L^\infty(BMO^{-1})$, then one can obtain a H\"older estimate on the solution $\theta$ by extending the methods of De Giorgi, Nash, and Moser (see \cite{FriedlanderVicol11a} or \cite{SSSZ12}). If the drift were not assumed to be divergence-free, then one would obtain H\"older continuity of the solution under the stronger assumption $u \in L^\infty(L^n)$. Note that the space $L^\infty(BMO^{-1})$ is larger than $L^\infty(L^n)$, but it has the same scaling. Therefore, the divergence-free assumption only provides a borderline improvement in this result.

In their celebrated paper \cite{CaffarelliVasseur10}, Caffarelli and Vasseur were able to prove well-posedness of the critical SQG equation based on an a priori estimate in H\"older spaces for \eqref{evo} when $s=1/2$ and $u$ is divergence-free and in $L^\infty(BMO)$ (well-posedness of SQG was also proved independently by Kiselev, Nazarov, and Volberg~\cite{KiselevNazarovVolberg07}).  Other proofs of this result were given in \cite{KiselevNazarov09} and~\cite{ConstantinVicol11}.
For non-divergence-free drifts, the same type of H\"older estimate was obtained in \cite{Silvestre2011} by a different method assuming $u \in L^\infty(L^\infty)$. Again, the divergence-free assumption only provides a borderline improvement in the regularity result since $BMO$ and $L^\infty$ have the same scaling properties.

In \cite{ConstantinWu09}, Constantin and Wu investigated lower powers of the Laplacian in the diffusion using  techniques from \cite{CaffarelliVasseur10} and \cite{CaffarelliSilvestre07}. They obtained a priori estimates in H\"older spaces for the equation \eqref{evo}
where $s \in (0,1/2)$, $u$ is divergence-free and  in $L^\infty(C^{1-2s})$. Using the ideas from \cite{Silvestre2011}, the result was generalized to non divergence-free drifts in \cite{Silvestre10b}. These estimates do not suffice to show well-posedness of the surface quasi-geostrophic equation in the supercritical regime. One might wonder whether the result in \cite{ConstantinWu09} could be improved using the divergence-free condition in a stronger way. In fact, in \cite{Chamorro10}, it was suggested that the solution of \eqref{evo} for any $s \in (0,1/2)$ would be H\"older continuous just assuming that $u \in L^\infty(BMO)$ and is divergence-free. We disprove this last statement here. We show that the result in \cite{ConstantinWu09} is in fact sharp by proving that for any $\alpha < 1-2s$, there is a divergence-free drift $u \in L^\infty(C^\alpha)$ for which the solution of \eqref{evo} develops a discontinuity starting from smooth initial data.

We now state our main results.  In this paper a {\it modulus of continuity} will be any continuous non-decreasing $\rho:[0,\infty)\to[0,\infty)$ such that $\rho(0)=0$, and we say that  $\theta:(t_0,t_1)\times\RR^d\to\RR$ {\it breaks (satisfies)}  $\rho$ at time $t$ if there are (no) $x,y\in\RR^d$ such that $|\theta(t,x)-\theta(t,y)|>\rho(|x-y|)$.  Although we will restrict our considerations to the case of two dimensions $d=2$ here, several of our results extend to more dimensions.

Our first result shows that the result of Constantin and Wu~\cite{ConstantinWu09} is sharp, even for time-independent drifts.

\begin{theorem}[\bf Case $s < 1/2$]\label{t:intro1}
Let $s\in(0,1/2)$ and $\alpha\in(0, 1-2s)$. There exist a positive time $T$ and a smooth function $\theta_0$ with $\|\theta_0\|_{C^2(\RR^2)}\le 1$ such that for any modulus of continuity $\rho$,  there exists a smooth divergence-free time-independent vector field $u$ with $\|u\|_{C^\alpha(\RR^2)} \le 1$ such that the smooth solution of \eqref{evo}--\eqref{evo:IC} breaks the modulus  $\rho$  before time $T$.
\end{theorem}

For the case $s \geq 1/2$, a critical assumption on the drift can be given in terms of Lebesgue spaces. It is conceivable that the method of Caffarelli and Vasseur~\cite{CaffarelliVasseur10} can be extended to $s \geq 1/2$ in dimension $d$ assuming that $u \in L^\infty(L^{d/(2s-1)})$, although this has not been written down anywhere, to the best of our knowledge. We prove that no weaker assumption on $u$ could assure continuity of the solution.

\begin{theorem}[\bf Case ${1/2 \leq s <1}$]\label{t:intro2}
Let $s \in [1/2,1)$ and $p \in[1, 2/(2s-1))$. There exist a positive time $T$ and a smooth function $\theta_0$ with $\|\theta_0\|_{C^2(\RR^2)}\le 1$ such that for any modulus of continuity $\rho$, there exists a smooth divergence-free time-independent vector field $u$ with $\|u\|_{L^p(\RR^2)}\le 1$ such that the smooth solution of \eqref{evo}--\eqref{evo:IC} breaks the modulus $\rho$ before time $T$.
\end{theorem}

The results obtained in Theorems~\ref{t:intro1} and \ref{t:intro2} for {\em smooth} drifts can be used to prove that there are divergence-free time-independent $u \in C^\alpha$ (with $\alpha<1-2s$, if $s\in(0,1/2)$), respectively $u \in L^p$ (with $p<2/(2s-1)$, if $s\in[ 1/2,1)$), with distributional solutions of the initial value problem \eqref{evo}--\eqref{evo:IC} evolving from smooth initial data, which fail to be continuous at any $t> 0$. 
Indeed, the drifts for different $\rho$s are constructed from a $\rho$-independent non-smooth drift using $\rho$-dependent cutoffs near the origin (where $\rho$ will be broken) to ensure smoothness.  Removing this cutoff at the origin will result in a (limiting as cutoff area shrinks to the origin) distributional solution which breaks any modulus $\rho$ in finite time.  Moreover, this is true for any time $t>0$ thanks to infinite speed of propagation of diffusion. 
{We believe that by following the ideas in \cite{Ambrosio05,DiPernaLions1989}, one can show that these distributional solutions are unique whenever the divergence free drift lies in
$L^1(BV)$ (and all the drifts considered in our Theorems~\ref{t:intro1}--\ref{t:intro3} have this regularity). 
}

One might ask what happens in the endpoint case of classical (and local) diffusion $s=1$, and the answer is quite intriguing.  First, we show that if we allow the drift to be time-dependent, then the above results continue to hold. {(We note that the remark after Theorem \ref{t:intro2} also remains valid in this case, this time after the removal of the temporal cutoff near the ``blow-up'' time $t_q>0$ from the proof, albeit with breaking of all moduli guaranteed only by time $t_q$.)}


\begin{theorem}[\bf Case $s=1$: time-dependent drifts]\label{t:intro3}
Let $s=1$ and $p\in[1, 2)$. There exist a positive time $T$ and a smooth function $\theta_0$ with $\|\theta_0\|_{C^2(\RR^2)}\le 1$ such that for any modulus of continuity $\rho$, there exists a smooth divergence-free  vector field $u$ with $\|\sup_{t} |u|\|_{L^p(\RR^2)}\le 1$ such that the smooth solution of \eqref{evo}--\eqref{evo:IC} breaks the modulus  $\rho$ before time $T$.
\end{theorem}

In the case of time-independent drifts, however, Theorem~\ref{t:intro3} is surprisingly false!  We prove that the solution $\theta$ has a logarithmic modulus of continuity which depends on $u$ via its local (supercritical) $L^1$ norm only and, in fact, continuous distributional solutions exist for non-smooth locally $L^1$ drifts.
This is a remarkable property which holds in two space dimensions only.

\begin{theorem}[\bf Case $s=1$: time-independent drifts]\label{t:intro4}  
Let $s=1$ 
and assume that $u \in L^1_{\rm loc}(\RR^2)$ with $\|u\|_{L^1_{\rm loc}(\RR^2)} = \sup_{x\in\RR^2} \|u\|_{L^1(B_1(x))}<\infty$ is a divergence-free time-independent vector field.  If $\theta_0\in C^2(\RR^2)\cap {W^{4,1}(\RR^2)}$, then there is a distributional solution  of \eqref{evo}--\eqref{evo:IC} which is continuous
and at any time $t>0$ satisfies  a  modulus of continuity given by
\begin{align}
\rho_t(r) =  \frac{{C(1+\|u\|_{L^1_{\rm loc}}) \|\theta_0\|_{C^2\cap W^{4,1}}} (1 + t^{-1})}{\sqrt{ - \log r}}
\label{e:MOC}
\end{align}
{for $r \in (0,1/2)$, with some universal $C>0$. }
\end{theorem}


{We remark that if instead of $u \in L^{1}_{\rm loc}$ we assume $u \in L^{1}$, then one may lower the regularity assumption on the initial data to $\theta_{0} \in C^{2} \cap W^{2,1}$ (see the proof).}

We note that the last claim in Lemma \ref{l:linfty-decay}  shows that for each $t_0>0$, this solution satisfies the (spatio-temporal) modulus $\rho_{t_0}$ on $(t_0,\infty)\times\RR^2$ as well.  Moreover, the result in fact holds for any distributional solution which is  a locally uniform limit of smooth solutions with drifts converging to $u$ in $L^1_{\rm loc}$. 
In particular, if $u$ is smooth, the (unique) solution of the Cauchy problem satisfies the modulus of continuity \eqref{e:MOC}, which depends on the drift via the super-critical norm $\|u\|_{L^1_{\rm loc}}$ only.  

An analogous result in the elliptic case  was proved in \cite{SSSZ12}: the elliptic maximum principle plus an a priori estimate in $H^1$, which hold for solutions of the  PDE, suffice to show that a function has a logarithmic modulus of continuity. {This idea can, in fact, be traced back to Lebesgue~\cite{Lebesgue1907}!} 

In the parabolic setting the situation is somewhat different. The parabolic maximum principle plus the energy estimates do not suffice to show the continuity of the solution. The following example illustrates the difficulty: if $\theta(t,x) = \varphi(x)$ for some $\varphi \in H^1(\RR^2) \setminus C(\RR^2)$, then $\theta \in C^\infty(H^1)$ and it satisfies the parabolic maximum principle, without being continuous. In order to overcome this difficulty, we need to use the equation to prove that for each time $t$, the elliptic maximum principle holds modulo an error that we can control (see Lemma \ref{l:elliptic-max}).  
A crucial ingredient will also be that $\partial_t\theta$ solves the same equation as $\theta$, which allows for some important bounds on this quantity (see Lemmas \ref{l:linfty-decay} and \ref{l:energy-ineq}).  Of course, this only holds when $u$ is time-independent.

Inspired by~\cite{DKSV12}, we also explore a {\em slightly supercritical} equation. This is the case in which the fractional Laplacian is replaced by an integral kernel which is logarithmically supercritical. 

\begin{theorem}[\bf A slightly supercritical case]\label{t:intro5}
Let $m:\RR^+ \to \RR^+$ be a smooth non-increasing function such that
\begin{align}
\int_0^\infty  \frac{m(r)}{1+r} \dd r < \infty \label{e:m}
\end{align}
and $rm(r)$ is non-decreasing on $(0,1)$.  
There exist a positive time $T$ and a smooth function $\theta_0$ with $\|\theta_0\|_{C^2(\RR^2)}\le 1$ such that for any modulus of continuity $\rho$, there exists a smooth divergence-free time-independent vector field $u$ with $\|u\|_{L^\infty(\RR^2)}\le 1$ such that the smooth solution of 
\begin{align}
\partial_t \theta + u \cdot \grad \theta + P.V. \int_{\RR^2} \Big(\theta(x) - \theta(x+y)\Big) \frac{m(|y|)}{|y|^2} \dd y &= 0 \label{e:nonlocal}
\end{align}
with initial condition $\theta_0$ 
 breaks the modulus  $\rho$ before time $T$.
\end{theorem}

This result suggests that in order to hope for continuity of solutions to \eqref{e:nonlocal}, one should not depart from the critical case $m(r) = 1/r$  by ``more than a logarithm''. It would be interesting to show that for generic divergence-free $L^\infty$ drifts, solutions to \eqref{e:nonlocal} are continuous if the integral in \eqref{e:m} diverges. In fact, for the dissipative Burgers equations it was shown in \cite{DKSV12} that when \eqref{e:m} holds, shocks develop in finite time, while if the integral in \eqref{e:m} diverges, global regularity holds.

The article is organized as follows. In Section~\ref{sec:Linfty} we present informally the main idea behind the loss continuity in finite time of solutions to  \eqref{evo} (for simplicity we take $u \in L^\infty$ and $s < 1/2$). Section~\ref{sec:Holder} contains the  proofs of Theorems~\ref{t:intro1} and \ref{t:intro2}, while Theorem~\ref{t:intro3} is proven in Section~\ref{sec:laplacian}. Sections~\ref{sec:trajectories} and \ref{sec:regularity} contain the proofs of the positive results in this paper, both for divergence-free time-independent drifts: uniqueness of particle trajectories for the transport equation with a continuous drift (Theorem~\ref{lemma:unique:X})
and Theorem~\ref{t:intro4}. We conclude  by proving Theorem~\ref{t:intro5} in Section~\ref{sec:barely}.

LS was supported in part by NSF grants DMS-1001629 and DMS-1065979, and an Alfred P. Sloan Research Fellowship.  AZ was supported in part by NSF grants DMS-1056327, DMS-1113017, DMS-1147523, and DMS-1159133, and an Alfred P. Sloan Research Fellowship. VV was supported in part by NSF grant DMS-1211828.

\section{Loss of regularity for bounded drift and supercritical dissipation} \label{sec:Linfty}

In this section we give a brief outline of our method of proving loss of continuity of solutions to drift-diffusion equations. In order to emphasize the main ideas, we present here the simplest case: the time-independent divergence-free drift is bounded and dissipation is super-critical with respect to the natural scaling of the equations (i.e., $\alpha=0$ and $s \in (0,1/2)$).  
For the sake of simplicity of exposition, we will not require here the drift to be smooth and simply {\it assume that we have a unique solution to \eqref{evo}--\eqref{evo:IC}}.  In the rigorous treatment in Section~\ref{sec:Holder}, we shall consider a smooth version of the drift defined 
below (and its generalization to other $\alpha$) so that we need not worry about global existence and regularity of solutions.

Define the stream function 
\[\psi(x_1,x_2) = \frac{|x_1-x_2| - |x_1+x_2|}{2}\]
and let $u = \nabla^\perp \psi$, where $\nabla^{\perp} = (-\partial_{x_2},\partial_{x_1})$. Then $u \in L^{\infty}(\RR^{2})$ is divergence-free (in the sense of distributions) and may be written explicitly as
\begin{figure}[htb!]
\begin{minipage}{0.5\linewidth}
\begin{align*}
u(x_1,x_2) = \begin{cases}
(0,-1) & \mbox{if } x_2>|x_1|,\\
(0,1) & \mbox{if } x_2<-|x_1|,\\
(1,0) & \mbox{if } x_1>|x_2|,\\
(-1,0) & \mbox{if } x_1<-|x_2|.
\end{cases} 
\end{align*}
\end{minipage}
\begin{minipage}{0.4\linewidth}
\includegraphics[scale=0.25]{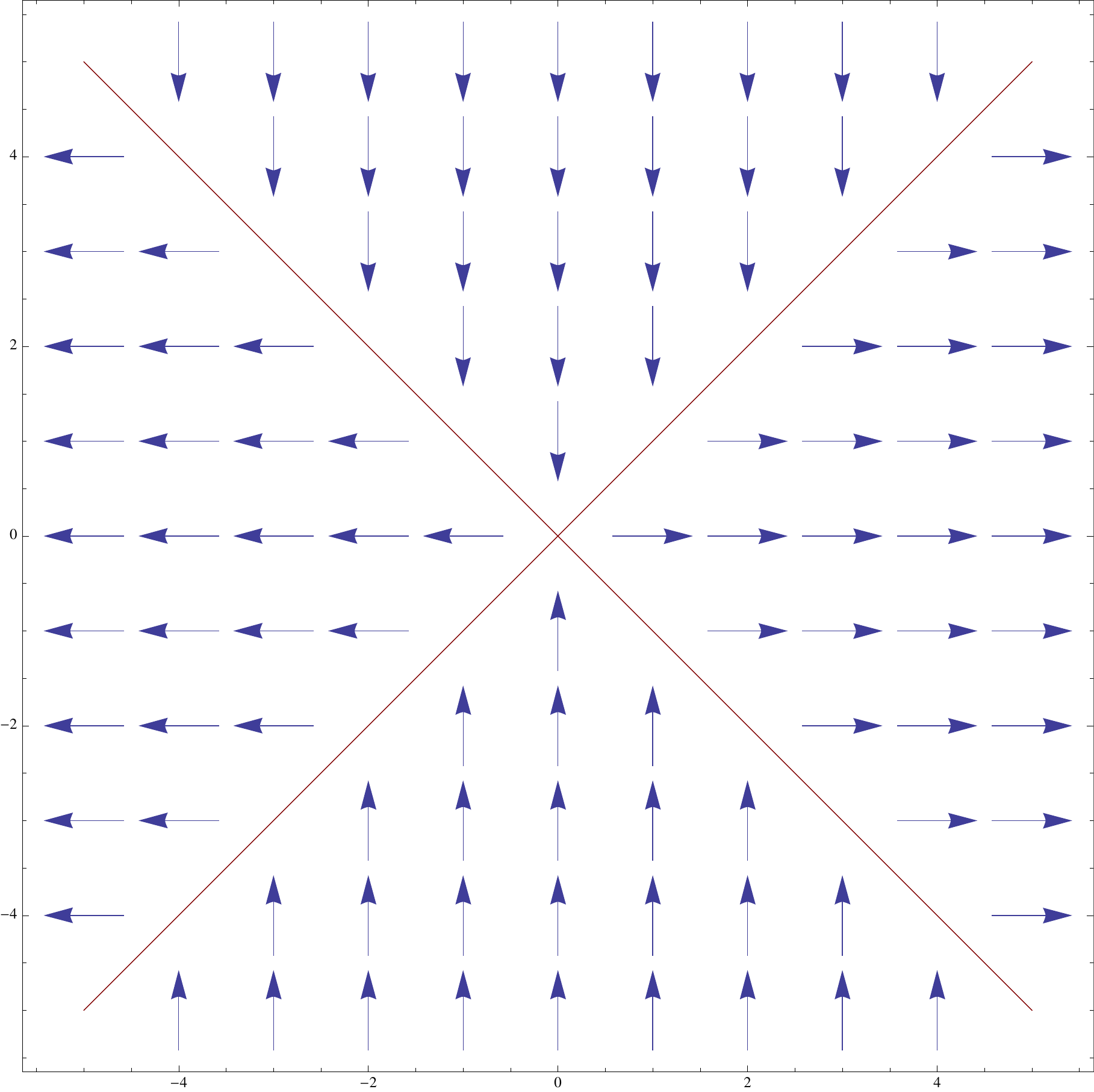}
\end{minipage}
\end{figure}
Let us consider \eqref{evo}--\eqref{evo:IC} with this drift, $s \in (0,1/2)$, and some smooth initial datum $\theta_0$.

\begin{theorem}[\bf Loss of continuity in the supercritical regime] \label{thm:L:infinity}
Let $s\in (0,1/2)$ and let $\rho$ be any modulus of continuity.
If $\theta_0$ is non-negative in the upper half-plane, larger than $1$ in $B_{1/2}(0,1)$, and odd in $x_2$, then the solution  of \eqref{evo}--\eqref{evo:IC} breaks the modulus  $\rho$ at some time $t_* \leq 1$. 
\end{theorem}

Since $\rho$ was arbitrary, it follows that solution  loses continuity in finite time, which may be viewed as a blow-up in $L^\infty((0,1);C(\RR^2))$.

Before proving Theorem~\ref{thm:L:infinity}, let us observe two key properties of the system \eqref{evo}--\eqref{evo:IC} (again, these hold in the smooth case and we {\it assume them to hold in the case at hand} as well).

\smallskip
\noindent{\bf Symmetry:} If  $\theta_0$ is odd in $x_2$, 
then so is the solution $\theta(t,\cdot)$ for $t>0$ (because $u(Rx)=Ru(x)$, where $R(x_1,x_2)=(x_1,-x_2)$).
\smallskip

\noindent{\bf Comparison principle:} If $\bar\theta$ is odd in $x_2$ and a super-solution of \eqref{evo} on the upper half-plane (and so a sub-solution on the lower half-plane) with  $\sgn(x_2)\bar\theta(0,x)\ge 0$, then for $t>0$ we have $\sgn(x_2)\bar \theta(t,x) \ge 0$. 
\smallskip

With these in mind, the intuition behind the proof of Theorem~\ref{thm:L:infinity} is as follows. Instead of \eqref{evo}--\eqref{evo:IC}, consider a pure transport equation with drift $u$, no dissipation, and odd-in-$x_2$ initial condition satisfying $\theta_0 \ge \chi_{B_{1/2}(0,1)}$ on the upper half-plane. Then the function 
\[ 
\bar \theta(t,\cdot) = \chi_{B_{(1-t)/2}(0,1-t)} - \chi_{B_{(1-t)/2}(0,t-1)}
\]
is a sub-solution, because the radius of the two discs forming its support decreases at the same rate as they approache the origin, so that the support stays in the set $|x_2|>|x_1|$.  So if $\theta$ is the actual solution, then $\theta-\bar\theta$ is an odd-in-$x_2$ super-solution on the upper half-plane.  The comparison principle (which also holds for the transport PDE) now shows that the oscillation of $\theta$ on $B_{1-t}(0,0)$, is at least $2$, so that any modulus is broken before time $t=1$.  Adding now  dissipation will decrease the supremum of $\theta(t,\cdot)$ on $B_{(1-t)/2}(0,1-t)$.  However, as the proof below shows,  the latter will stay bounded away from zero on any finite time interval as long as the equation is supercritical. 

\begin{proof}[Proof of Theorem~\ref{thm:L:infinity}]
Let $\eta\in C^\infty(\RR)$ be an even function, {supported on $[-1,1]$, positive on $(-1,1)$}, non-increasing on $\RR^+$, and with $\eta(0)=1$.   For some $r \in (0,1/\sqrt 2)$, we let
\begin{align*} 
\phi(x_1,x_2) = \eta(|(x_1,x_2-1)| r^{-1}) - \eta (|(x_1,x_2+1)| r^{-1}),
\end{align*}
so that $\phi$ consists of a positive smooth bump centered at $(0,1)$ and with radius $r$, and a similar negative bump centered at $(0,-1)$.  Then for any $s \in (0,1/2)$ there is $c_{r,s}>0$ such that
\begin{align} 
(-\Delta)^s \phi(x_1,x_2) \leq c_{r,s} \phi(x_1,x_2) \label{eq:sub:eig:bdd}
\end{align}
holds in the upper half-plane $\{ x_2 > 0\}$ (the proof of this is given in Lemma~\ref{lemma:dissipation} below).

We now let $\theta_0=\phi$ and 
\begin{align} 
\bar \theta (t,x_1,x_2) = \exp\left(- c_{r,s} \int_0^t  z(\tau)^{-2s} \dd \tau \right) \phi\left( \frac{x_1}{z(t)}, \frac{x_2}{z(t)}\right) \label{eq:barrier:bdd}
\end{align}
for $t\in[0,1)$, where 
\begin{align} 
z(t) = 1- t. \label{eq:z:bdd}
\end{align}
Notice that $z$ is the solution of the ODE
\[
\dot z(t) = u_2(0,z(t)) =-1, \qquad z(0)=1,
\]
which is the position of the original center $(0,1)$ of the positive bump, transported by the drift $u$.
It is clear that $\bar \theta \geq 0$ for $x_2>0$ and is odd in $x_2$.  The support of $\bar \theta(t,\cdot)$ consists of two discs whose centers $(0,\pm z(t))$ are transported towards the origin with the drift $u$, and have radii $r  z (t)$. That is, the support shrinks at the same rate it approaches the origin (and lies in the set $|x_2|>|x_1|$ because $r<1/\sqrt 2$).

We now claim that $\bar\theta$ is a sub-solution of \eqref{evo} in $\{ x_2>0\}$. By scaling and \eqref{eq:sub:eig:bdd} we have
\begin{align*} 
(-\Delta)^s \bar \theta(t,\cdot) \leq c_{r,s} z(t)^{-2s} \bar \theta(t,\cdot)
\end{align*}
in $\{x_2>0\}$, so it suffices to prove
\begin{align} 
\partial_t \bar\theta + u \cdot \nabla \bar\theta +  c_{r,s} z(t)^{-2s} \bar \theta \leq 0\label{eq:transport:sub:bdd}
\end{align} 
in $\{x_2>0\}$.
From the definition of $\bar \theta$, it is clear that we just need to verify
\[
\partial_t \phi_+ + u \cdot \nabla \phi_+ \leq 0, \qquad \mbox{where} \qquad \phi_+(t,x_1,x_2)=  \eta \left(\frac{|(x_1,x_2-z(t))|}  {r z(t)} \right).
\]

The point now is that $\phi_+$ is transported by the vector field $v(t,x)=-x/z(t)$, that is, it satisfies
\[
\partial_t \phi_+ + v \cdot \nabla \phi_+ = 0,
\]
as one can see by a simple computation. 
We therefore only need to show $(v-u)\cdot \nabla \phi_+\ge 0$ which, due to $\phi_+$ being supported in $x_2>|x_1|$, radially symmetric, and radially non-increasing with respect to center $(0,z(t))$, is equivalent to
\begin{equation} \label{1.30}
0 \leq  ( (x_1,x_2) z(t)^{-1} +  (0,-1) )  \cdot (x_1,x_2-z(t)) = z(t)^{-1} |(x_{1}, x_2 - z(t))|^2
\end{equation}
on the support of $\phi_+$. This clearly holds, so $\bar\theta$ is a sub-solution of \eqref{evo} in $\{x_2>0\}$.

To conclude the proof, note that by the comparison principle, \eqref{eq:z:bdd} and \eqref{eq:barrier:bdd}, we have
\begin{align*} 
\osc_{B_{ z(t)}(0)}\theta(t,\cdot)  \geq \osc_{B_{ z(t)}(0)} \bar \theta(t,\cdot) = 2 \exp\left(-  c_{r,s} \int_0^t (1-\tau)^{-2s} \dd \tau \right) \geq 2 \exp \left( - \frac{ c_{r,s}}{1-2s}\right) .
\end{align*}
if $s \in (0,1/2)$ and $t\in [0,1)$. Thus, if the solution $\theta$ of \eqref{evo}--\eqref{evo:IC} obeyed the modulus of continuity $\rho$ for all $t\in[0,1)$, then we would have $\rho(0)=\lim_{t\to 1} \rho(2-2t) >0$,  a contradiction.
\end{proof}

\section{The case \texorpdfstring{$s\in(0,1)$}{0<s<1} with supercritical drift} \label{sec:Holder}

The main ideas for  finite time loss of continuity of solutions to \eqref{evo}--\eqref{evo:IC} in the case of bounded drifts and $s<1/2$ were presented in Section~\ref{sec:Linfty}. Here we extend those arguments to treat all values of $s \in (0,1)$ and corresponding  supercritical drifts. For $\alpha \in (-1,1)$ we denote
\begin{align}
X^\alpha:= \begin{cases}
C^\alpha(\RR^2) & \mbox{for } \alpha \in (0,1),\\
L^\infty(\RR^2) & \mbox{for } \alpha = 0,\\
L^{2/|\alpha|}(\RR^{2}) & \mbox{for }\alpha \in (-1,0). 
\end{cases} \label{eq:Xalpha}
\end{align}
In view of the natural scaling of the equations, for any $\alpha \in (-1,1-2s)$ the above  Banach space $X_\alpha$ is supercritical.

\begin{theorem}[\bf Finite time blow-up in supercritical regime] \label{thm:Holder}
Let $s \in (0,1)$, $\alpha' \in (-1,1-2s)$, and let $\rho$ be any modulus of continuity.  Then there exist a smooth divergence-free time-independent vector field $u$ with $\|u\|_{X^{\alpha'}}\le 1$ and a smooth function $\theta_0$ with $\|\theta_0\|_{C^2(\RR^2)}\le 1$ such that the smooth solution of \eqref{evo}--\eqref{evo:IC} breaks the modulus  $\rho$ in finite time, bounded above independently of $\rho$.
\end{theorem}


Notice that this result contains Theorems \ref{t:intro1} and \ref{t:intro2}.  It does not directly cover the case  $p\in[1,2]$ in Theorem \ref{t:intro2} but this follows from the case $p\in(2,2/(2s-1))$.  This is because our vector fields will be supported in $B_4(0,0)$ 
and because the bound $\|u\|_{X^{\alpha'}}\le 1$ can be replaced by $\|u\|_{X^{\alpha'}}\le c$ for any $c>0$ due to $M$ in \eqref{eq:ualpha} being arbitrarily large. 

Let us now turn to the construction of the drift $u$.  
For $\alpha \in (-1,1)$, define the stream function 
\begin{align}
\psi_{\alpha,s,\rho}(x_1,x_2) = \frac{|x_1-x_2|^{1+\alpha} - |x_1+x_2|^{1+\alpha}}{2(1+\alpha)}  \kappa\left( \frac{x_1}{x_2} \right) \mu_{\alpha,s,\rho} \left( | x|\right) 
\label{eq:stream}
\end{align}
where the last two factors are smooth cutoff functions ($\kappa$ is also even) designed to remove singularities at $|x_1|=|x_2|$ and $|x|=0$.  Specifically, $\kappa,\mu_{\alpha,s,\rho}\in C_0^\infty(\RR)$ are such that
\[
\chi_{[-1/2,1/2]} \le \kappa \le \chi_{[-2/3,2/3]}
\]  
and
\[ \chi_{ [2\eps_{\alpha,s,\rho},2]} \le \mu_{\alpha,s,\rho} \le \chi_{ [\eps_{\alpha,s,\rho},3]}, 
\]
with $\kappa$ even, $\eps_{\alpha,s,\rho}>0$ to chosen later, and also
\begin{equation} \label{eq:cutoff}
|\mu_{\alpha,s,\rho}'(z)| \le \frac 3{|z|}  
\end{equation}
(which is possible for any choice of $\eps_{\alpha,s,\rho}>0$).
Here $\kappa$ removes the singularities at $|x_1|=|x_2|$ (except the origin) but does not alter the fraction in \eqref{eq:stream} on the union of two cones given by $|x_1|\le \tfrac 12 |x_2|$.  The subsolution we will use, similar to that in \eqref{eq:barrier:bdd}, will be supported in this set, so $\kappa$ will not affect the argument.   Likewise, $\mu_{\alpha,s,\rho}$ removes the singularity at the origin and will make our drift compactly supported.  Since the subsolution will also be supported in the annulus $B_2(0)\setminus B_{2\eps_{\alpha,s,\rho}}(0)$ at all times until $\rho$ is broken, $\mu_{\alpha,s,\rho}$ will also not affect the argument from the previous section.  In fact, $\eps_{\alpha,s,\rho}$ will be chosen so that the modulus $\rho$ will be broken before the support of the subsolution reaches $B_{2\eps_{\alpha,s,\rho}}(0)$.

We now let, for some $M>0$,
\begin{align}
u_{\alpha,s,\rho} = \frac1M \nabla^\perp \psi_{\alpha,s,\rho} \label{eq:ualpha},
\end{align}
which is smooth and compactly supported because so is $\psi_{\alpha,s,\rho}$.  Then \eqref{eq:cutoff} ensures that for any given $\alpha\in[0,1)$ there is $M>0$ such that $\|u_{\alpha,s,\rho}\|_{X^\alpha}\le 1$ for all $s,\rho$, while for any given $\alpha\in(-1,0)$ and any $\alpha'\in(-1,\alpha)$, there is $M>0$ such that $\|u_{\alpha,s,\rho}\|_{X^{\alpha'}}\le 1$ for all $s,\rho$ (the latter because if $\alpha<0$, then $|\nabla^\perp\psi_{\alpha,s,\rho}(x)|\le C_\alpha|x|^\alpha$ with $C_\alpha$ independent of $s,\rho$).  If now $s\in(0,1)$ and $\alpha' \in (-1,1-2s)$ are given, we either pick $\alpha=\alpha'$ and the $M$ associated with $\alpha$ (if $\alpha'\in[0,1)$) or pick some $\alpha\in(\alpha',\min\{1-2s,0\})$ and the $M$ associated with $\alpha,\alpha'$ (if $\alpha'\in(-1,0)$).  In either case, the drift from \eqref{eq:ualpha} will satisfy $\|u_{\alpha,s,\rho}\|_{X^{\alpha'}}\le 1$ for any $\rho$ (and $M=M_{\alpha,s}$ is independet of $\rho$).

\begin{figure}[htb!]
\begin{minipage}{0.47\linewidth}
\begin{center}
\setlength{\unitlength}{1.5in} 
\begin{picture}(2.02083,2)
\put(0,0){\includegraphics[height=3in]{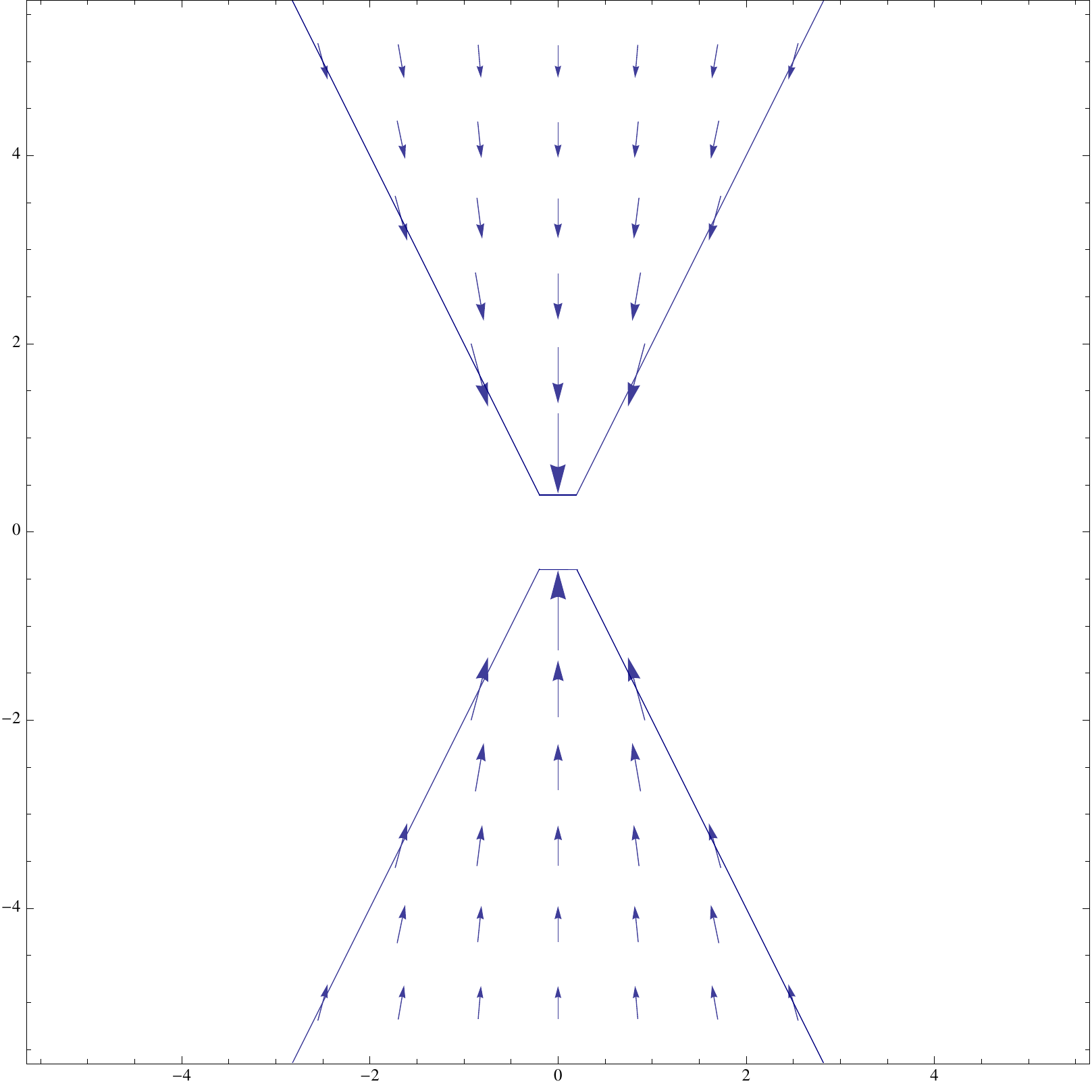}}
\put(1.7,0.4){\mbox{\rotatebox{90}{\tiny The values of $u$ are irrelevant here}}}
\put(0.3,1.6){\mbox{\rotatebox{270}{\tiny The values of $u$ are irrelevant here}}}
\end{picture}
\caption{Velocity field for $\alpha = -1/2$.}
\end{center}
\end{minipage}
\begin{minipage}{0.47\linewidth}
\begin{center}
\setlength{\unitlength}{1.5in} 
\begin{picture}(2.02083,2)
\put(0,0){\includegraphics[height=3in]{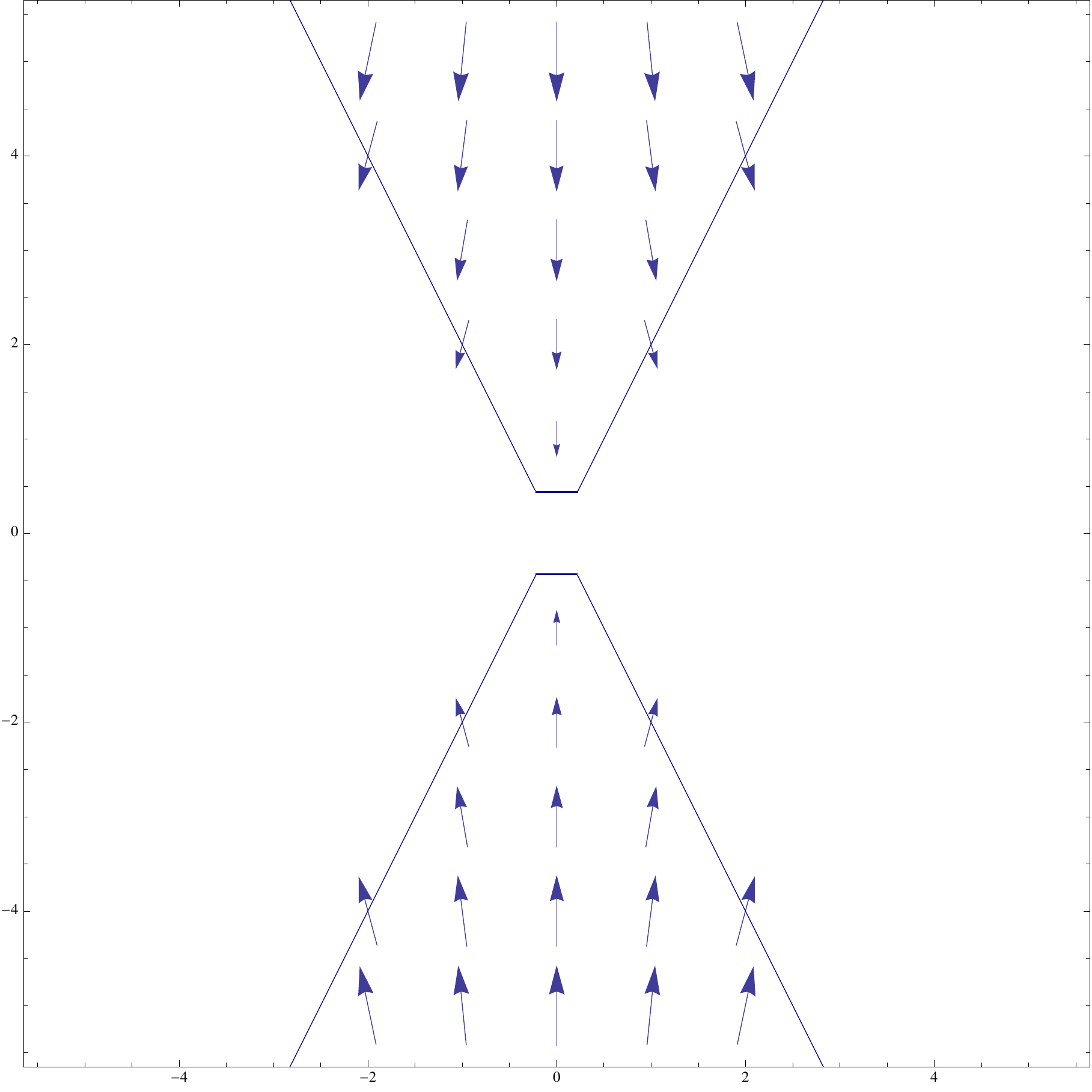}}
\put(1.7,0.4){\mbox{\rotatebox{90}{\tiny The values of $u$ are irrelevant here}}}
\put(0.3,1.6){\mbox{\rotatebox{270}{\tiny The values of $u$ are irrelevant here}}}
\end{picture}
\caption{Velocity field for $\alpha = 1/2$.}
\end{center}
\end{minipage}
\end{figure}

We also note the explicit formula
\begin{align}
u_{\alpha,s,\rho}(x_1,x_2) = \frac{\sign(x_2)}{2M} \left( - |x_1-x_2|^{\alpha} + |x_1+x_2|^{\alpha}, - |x_1-x_2|^{\alpha} - |x_1+x_2|^{\alpha} \right)\label{eq:u:Holder:def}
\end{align}
for all $x \in \CC_{\alpha,s,\rho} = \{x\in\RR^2 \,\big |\, |x_1|\le\tfrac 12 |x_2| \text{ and } |x|\in[2\eps_{\alpha,s,\rho},2]\}$. In particular,  we have
\begin{align} \label{1.2}
u_{\alpha,s,\rho}(0,x_2)=  M^{-1} (0,-x_2^\alpha)
\end{align}
for $x_2 \in [2\eps_{\alpha,s,\rho},2]$. 

Finally, note that since $u_{\alpha,s,\rho}$ is  smooth and compactly supported, the solution $\theta$ of \eqref{evo}--\eqref{evo:IC} will be  smooth and decaying at spatial infinity uniformly in $t\in[0,T]$ for any $T>0$ as long as  $\theta_0$ is smooth and decaying at infinity.  We will assume this from now on.  In what follows, we will also drop the subscripts and denote $u=u_{\alpha,s,\rho}$,  $\eps=\eps_{\alpha,s,\rho}$, and $\CC=\CC_{\alpha,s,\rho}$.

\begin{lemma}[\bf Symmetry] \label{lemma:symmetries}
If  $\theta_0$ is odd in $x_2$, 
then so is the solution $\theta(t,\cdot)$ of \eqref{evo}--\eqref{evo:IC} for $t>0$.
\end{lemma}

\begin{proof}
This follows from uniqueness of the solution and from  $u(Rx)=Ru(x)$, which is due to $\psi_{\alpha,s,\rho}$ being odd in $x_2$.
Here and throughout the paper we denote $R(x_1,x_2)=(x_1,-x_2)$.
\end{proof}

The symmetry of $u$ also shows that if an odd-in-$x_2$ function $\theta$ is  a sub(super)-solution of \eqref{evo} in the upper half-plane $\{x_2>0\}$, then it is a super(sub)-solution in the lower half-plane. In view of Lemma \ref{lemma:symmetries}, it is therefore natural to compare odd-in-$x_2$ solutions with odd-in-$x_2$ functions that are  sub(super)-solutions on $\{x_2>0\}$ only.  From now on we will call such functions {\it odd-in-$x_2$ sub(super)-solutions}, and for them we have:

\begin{lemma}[\bf Comparison principle] \label{lemma:comparison}
If $\bar\theta$ is  smooth,  decaying at spatial infinity uniformly in $t\in[0,T]$ for any $T>0$, an odd-in-$x_2$ super-solution of \eqref{evo},  and satisfies $\sgn(x_2)\bar\theta(0,x)\ge 0$, then for $t>0$ we have $\sgn(x_2) \bar\theta(t,x) \ge 0$. 
\end{lemma}

\begin{proof}
Due to oddness in $x_2$ and smoothness, we have $\bar\theta(t,x_1,0) =0$ for $t>0$. Assume, towards contradiction, that for some $T>0$ we have
\[ 
\inf_{x_1 \in \RR, x_2>0, t \in [0,T]} \bar\theta (t,x_1,x_2)  < 0.
\]
Since $\bar\theta$ is smooth, odd in $x_2$, decays at spatial infinity uniformly in $t\in[0,T]$, and $\bar\theta(0,\cdot) \geq 0$, the above infimum must be attained at some point $(t,x)$, with $t  \in (0,T]$ and $x_2 >0$. We then must have  {$\partial_t \bar\theta (t,x)\leq 0$} and $\nabla \bar\theta (t,x)= 0$, hence $(-\Delta)^s \bar\theta (t,x)\ge 0$. 

On the other hand, oddness in $x_2$ yields
\begin{align*} 
(-\Delta)^s \bar\theta(t,x) &= c_s \int_{\RR^2}  \frac{ \bar\theta(t,x) - \bar\theta(t,y)}{|x-y|^{2+2s}} \dd y \notag\\
&= c_s \int_{\{y_2>0\}}  \left( \frac{ \bar\theta(t,x) - \bar\theta(t,y)}{|x-y|^{2+2s}} + \frac{ \bar\theta(t,x) + \bar\theta(t,y)}{|x-Ry|^{2+2s}} \right) \dd y
\end{align*}
(recall that $Ry=(y_1,-y_2)$). For any $y$ in the upper half-plane we have  $|x-y| \leq |x-Ry|$ because $x_2>0$, and since the minimum of $ \bar\theta$ occurs at $x$, we also know $\bar\theta(t,x) - \bar\theta(t,y) \leq 0$. Hence
\[
 \frac{ \bar\theta(t,x) - \bar\theta(t,y)}{|x-y|^{2+2s}}  + \frac{ \bar\theta(t,x) + \bar\theta(t,y)}{|x-Ry|^{2+2s}} \leq \frac{ \bar\theta(t,x) - \bar\theta(t,y)}{|x-Ry|^{2+2s}}  + \frac{ \bar\theta(t,x) + \bar\theta(t,y)}{|x-Ry|^{2+2s}} = \frac{ 2 \bar\theta(t,x)}{|x-Ry|^{2+2s}}
\]
and therefore $(-\Delta)^s \bar\theta(t,x) \leq 2 c_s  \bar\theta(t,x) \int_{\{y_2>0\}} |x-Ry|^{-2-2s} \dd y < 0$, a contradiction. 
\end{proof}

 The above lemma will be applied to the difference between an odd-in-$x_2$ solution $\theta$ and an odd-in-$x_2$ sub-solution $\theta_{\alpha,s}$ of \eqref{evo}, so that $\theta-\theta_{\alpha,s}$ is an odd-in-$x_2$ super-solution. If the latter is non-negative in the upper half-plane at $t=0$, it follows from Lemma~\ref{lemma:comparison} that at all later times, the oscillation of $\theta$ over any disc centered at the origin is no less than the oscillation of $\theta_{\alpha,s}$ over the same disc. 

We now turn to the construction of $\theta_{\alpha,s}$, in the spirit of the argument in Section~\ref{sec:Linfty}. Let $\eta\in C^\infty(\RR)$ be an even function, {supported on $[-1,1]$, positive on $(-1,1)$}, non-increasing on $\RR^+$, and with $\eta(0)=1$.    For some $r_\alpha \in (0,1/4)$ (to be chosen later), we let
\begin{align} 
\phi_\alpha(x_1,x_2) = \eta(|(x_1,x_2-1)| r_\alpha^{-1}) - \eta (|(x_1,x_2+1)| r_\alpha^{-1}) \label{eq:phi:alpha}
\end{align}
be a smooth approximation of $\chi_{B_{r_\alpha}(0,1)} - \chi_{B_{r_\alpha}(0,-1)}$. As in Section~\ref{sec:Linfty}, we shall use $\phi_\alpha$ to build $\theta_{\alpha,s}$, but first 
we give a measure of the effect of $(-\Delta)^{s}$ on $\phi_\alpha$. 

\begin{lemma}[\bf Control of dissipation] \label{lemma:dissipation}
 Let $\alpha \in (-1,1)$, $r_\alpha \in (0,1/4)$, and $\phi_\alpha$  from \eqref{eq:phi:alpha}. Then there exists  $c_{\alpha,s}>0$ such that 
\begin{align} \label{1.3}
(-\Delta)^s\phi_\alpha(x) \leq c_{\alpha,s}\phi_\alpha(x)
\end{align}
holds in the upper half-plane $\{x_2>0\}$. The constant $c_{\alpha,s}$ depends only on $r_\alpha,s $, and $\eta$.
\end{lemma}
\begin{proof}[Proof of Lemma~\ref{lemma:dissipation}]
We start with $x \in \{x_2>0\} \setminus B_{r_\alpha}(0,1)$. Due to the definition of $\eta$, we have $\phi_\alpha(x)=0$, and using oddness in $x_2$ we obtain
\begin{align} 
(-\Delta)^s \phi_\alpha(x) &= - c_s \int_{\RR^{2}} \frac{ \phi_\alpha(y)}{|x-y|^{2+2s}} \dd y\notag\\
&= - \int_{B_{r_\alpha}(0,1)}  \frac{\eta( |(y_1,y_2-1)|r_\alpha^{-1})}{|x-y|^{2+2s}} \dd y + \int_{B_{r_\alpha}(0,-1)} \frac{\eta( |(y_1,y_2+1)|r_\alpha^{-1})}{|x-y|^{2+2s}} \dd y  \notag\\
& = - \int_{B_{r_\alpha}(0,1)} \eta(|(y_1,y_2-1)|r_\alpha^{-1}) \left( \frac{1}{|x-y|^{2+2s}} - \frac{1}{|x-Ry|^{2+2s}} \right) \dd y.
\label{eq:diss:1}
\end{align}
Now we notice that for $x,y$ in the upper half-plane we have  $|x-y| \leq |x-Ry|$, and hence the integrand in \eqref{eq:diss:1} is positive. It follows that  $(-\Delta)^s \phi_\alpha(x) < 0 $ for all $x \in \{x_2>0\} \setminus B_{r_\alpha}(0,1)$.
Since  $\partial B_{r_\alpha}(0,1)$ is compact and $\phi_\alpha \in C^\infty$, it follows that there is $c_1 = c_1(r_\alpha,s,\eta)>0$ such that 
\begin{align*}
 (-\Delta)^s \phi_\alpha(x)\leq - c_1 < 0
 \end{align*}
for all $x \in \partial B_{r_\alpha}(0,1)$. Hence there exists $\delta =\delta(r_\alpha,s,\eta) \in (0,r_\alpha/2)$ such that $(-\Delta)^s \phi_\alpha(x) \leq 0$ for all $x \in \{x_2>0\} \setminus B_{r_\alpha-\delta}(0,1)$, implying \eqref{1.3} for these $x$ (with any $c_{\alpha,s}$).


It is left to verify \eqref{1.3} for all $x \in B_{r_\alpha - \delta}(0,1)$.  in view of the monotonicity of $\eta$, for each such $x$ we have that $\phi_\alpha(x) \geq  \eta(1-\delta r_\alpha^{-1}) > 0$. Therefore, for these $x$ we have 
\begin{align*}
(-\Delta)^s \phi_\alpha(x) &\leq c_2 r_\alpha^{-2s} \|\eta\|_{C^{2}} \leq c_{\alpha,s} \eta(1 -\delta r_\alpha^{-1}) \leq c_{\alpha,s} \phi_\alpha(x)
\end{align*}
for a suitably chosen $c_{\alpha,s} > 0$ (depending on $r_\alpha,s,\delta,\eta$), which concludes the proof.
\end{proof}

From  homogeneity of the kernel associated with the fractional Laplacian $(-\Delta)^s$ and from Lemma~\ref{lemma:dissipation} we directly have the following.

\begin{corollary}[\bf Effect of dissipation under rescaling]\label{cor:dissipation}
For $\lambda > 0$ let $\phi_{\alpha,\lambda}(x) = \phi_\alpha(x/\lambda)$.  Then
\begin{align}
(-\Delta)^s \phi_{\alpha,\lambda}(x)  \leq c_{\alpha,s} \lambda^{-2s} \phi_{\alpha,\lambda}(x)
\label{eq:diss:rescale}
\end{align}
holds on $\{x_2>0\}$.
\end{corollary}

Having established Lemma~\ref{lemma:dissipation} and Corollary~\ref{cor:dissipation}, it is clear from the argument in Section~\ref{sec:Linfty} how {we will} proceed.  We will build an odd-in-$x_2$ sub-solution to \eqref{evo} by adding an appropriate exponentially decaying prefactor (chosen using \eqref{eq:diss:rescale}) to an odd-in-$x_2$ sub-solution of the pure transport equation $(\partial_t + u \cdot \grad) \phi(t,x)=0$ with initial condition $\phi_\alpha$.  This sub-solution will be a time-dependent rescaling of $\phi_\alpha$, that will depend on the vector field $u=u_{\alpha,s,\rho}$ defined in \eqref{eq:ualpha}.  Note that it is easier to build a sub-solution then to find the actual solution because the disks on which $\phi_\alpha$ is supported deform when transported by  $u$. 


Let $T=T_{\alpha,s,\rho} = M(1-(4\eps)^{1-\alpha})/(1-\alpha)$ (recall that $\eps=\eps_{\alpha,s,\rho}$) and for $[0,T]$ let
\begin{align}\label{1.4}
z_\alpha(t)=\left(1-(1-\alpha)M^{-1}t\right)^{1/(1-\alpha)}.
\end{align}
Recalling \eqref{1.2}, we find that $z_\alpha$ is  the solution of the initial value problem
\begin{align} \label{1.5}
\dot{z}(t)=u_{2}(0,z(t)) = -M^{-1}z(t)^\alpha, \qquad z(0)=1,
\end{align}
because it is decreasing and $z_\alpha(T) = 4\eps\ge 2\eps$.  

\begin{lemma}[\bf Subsolution of the transport equation] \label{lemma:sub:transport}
The function $\bar \phi_\alpha(t,x)=\phi_{\alpha}(x/z_\alpha(t))$ 
is an odd-in-$x_2$ sub-solution of $(\partial_t + u_{\alpha,s,\rho} \cdot \nabla)\phi=0$ 
on the time interval $(0,T_{\alpha,s,\rho})$, provided $r_\alpha \in (0,1/4)$ is sufficiently small. 
\end{lemma}

\begin{proof}[Proof of Lemma~\ref{lemma:sub:transport}] By the definition of $\phi_{\alpha}$ it suffices to prove
\[
\partial_t \bar \phi +u \cdot\nabla \bar\phi\leq 0
\]
where 
\[
\bar \phi(t,x_1,x_2) = \eta\left(\frac{|(x_1,x_2-z_\alpha(t))| }{r_\alpha z_\alpha(t)}\right)
\]
is the component of $\phi_\alpha(x/z_\alpha(t))$ which is supported in the upper half-plane.
Let 
\[
v(t,x_1,x_2)= - M^{-1}z_\alpha(t)^{\alpha-1}(x_1,x_2)
\] 
be the vector field that advects $\bar \phi$, so that we have
$
\partial_t \bar \phi +v \cdot\nabla\bar \phi = 0.
$
Hence, in order to prove the lemma  it suffices to show
\[
\left( v(t,x_1,x_2)-u(x_1,x_2) \right) \cdot \nabla \bar\phi(t,x_1,x_2) \ge 0
\]
which, since $\eta$ is non-increasing on $\RR^+$, is equivalent to
\begin{align} \label{1.8}
\left( u(x_1,x_2)+M^{-1}z_\alpha(t)^{\alpha-1}(x_1,x_2) \right) \cdot (x_1,x_2-z_\alpha(t)) \ge 0.
\end{align}
It suffices to prove \eqref{1.8} for $t\in[0,T]$ and $x_1^2+(x_2-z_\alpha(t))^2\le (r_\alpha z_\alpha(t))^2$ (notice that such $(x_1,x_2)$ then lie in the domain $\CC$ because $r_\alpha<1/4$ and $z(t)\in[4\eps,1]$). By $\alpha$-homogeneity of $u$ in  $\CC$ (see~\eqref{eq:u:Holder:def}) and after scaling by $z_\alpha(t)^{-1}$,  it is sufficient to verify that
\begin{align} \label{1.9}
\left( Mu(x_1,x_2)+(x_1,x_2) \right) \cdot (x_1,x_2-1) \ge 0
\end{align}
holds for $x_1^2+(x_2-1)^2\le r_\alpha ^2$.
Since \eqref{eq:u:Holder:def} holds for these $x$, the {left-hand} side of \eqref{1.9} there equals
\[
f(x_1,x_2)=-\frac 12 (x_2-x_1)^{\alpha}(x_2-1+x_1) -\frac 12 (x_2+x_1)^{\alpha}(x_2-1-x_1) + x_2(x_2-1)+x_1^2.
\]
After an elementary 
computation we obtain
\begin{gather}
f(0,1)=f_{x_1}(0,1)=f_{x_2}(0,1)=f_{x_1x_2}(0,1)=0,
\\ f_{x_1x_1}(0,1)=2(1+\alpha)>0, \quad f_{x_2x_2}(0,1)=2(1-\alpha)>0,
\end{gather}
and therefore $(0,1)$ is a local minimum of $f$ for any $\alpha \in (-1,1)$.
Since $f$ is $C^2$ near $(0,1)$, there exists $r_\alpha\in(0,1/4)$ such that  $f(x_1,x_2)\ge 0$ for $x_1^2+(x_2-1)^2\le r_\alpha ^2$, which proves \eqref{1.9}, and hence the lemma.
\end{proof}

\begin{proof}[Proof of Theorem~\ref{thm:Holder}]
We now let
\begin{align} \label{1.6}
\gamma_{\alpha,s}(t)=  \int_0^t c_{\alpha,s}z_\alpha(\tau)^{-2s} \dd \tau = 
\begin{cases}
c_{\alpha,s}M(1-\alpha-2s)^{-1} \left( 1- z_\alpha(t)^{1-\alpha-2s} \right) & \alpha+2s < 1, \\
- c_{\alpha,s}M\ln z_\alpha(t) & \alpha+2s= 1, 
\end{cases}
\end{align}
where $c_{\alpha,s}$ is  from \eqref{1.3}, and $t\in [0,T]$.
Finally, we define a rescaled modulated version of $\phi_\alpha$:
\begin{align}\label{1.7}
\theta_{\alpha,s}(t,x)= c \exp\left( -  \gamma_{\alpha,s}(t)\right) \phi_\alpha\left(\frac{x}{z_\alpha(t)}\right),
\end{align}
where $c=\|\phi_\alpha\|_{C^2}^{-1}$.
Again, $\theta_{\alpha,s}$ is a function supported at any time $t$ on two discs whose centers $(0,\pm z_\alpha(t))$ are transported towards the origin with the drift $u$ and whose radii are $r_\alpha z_\alpha(t)$, that is, they shrink at the same rate as they approach the origin.

 Lemma~\ref{lemma:sub:transport} and  Corollary~\ref{cor:dissipation}  together show that $\theta_{\alpha,s}$  is an odd-in-$x_2$ sub-solution of \eqref{evo} for $t\in[0,T]$ because then
 \begin{align*}
\left( \partial_t + u \cdot \nabla + (-\Delta)^s \right) \theta_{\alpha,s} \leq \dot{\gamma}_{\alpha,s}(t) \theta_{\alpha,s} + c_{r,s} z_\alpha(t)^{-2s} \theta_{\alpha,s} = 0
\end{align*} 
holds on $\{x_2>0\}$.  If $\theta$ is the odd-in-$x_2$ solution with initial condition $\theta_0=c\phi_\alpha(=\theta_{\alpha,s}(0,\cdot))$, Lemma \ref{lemma:comparison} shows that $\sgn(x_2)\theta\ge \sgn(x_2)\theta_{\alpha,s}$ for $t\in[0,T]$. We now only need to pick $\eps=\eps_{\alpha,s,\rho}>0$ so that $\theta_{\alpha,s}$ (and thus also $\theta$) breaks $\rho$ at the origin at time $T=T_{\alpha,s,\rho}$.  This will be possible because if we picked $\eps=0$ and thus $T=M/(1-\alpha)$ (recall that $\alpha+2s<1$, so $z_\alpha(M/(1-\alpha))=0$), then we would obtain  $\gamma_{\alpha,s}(T) <\infty$, so $\theta$ would become discontinuous in finite time.  

Of course, we will need $\eps>0$ to ensure smoothness of the drift $u$ and the solution $\theta$. We let $\eps=\eps_{\alpha,s,\rho}>0$ be such that 
%
\begin{align}
\rho(8 \eps) < 2c\exp(-\gamma_{\alpha,s}(M/(1-\alpha)))=2c\exp\left( - c_{\alpha,s} (1-\alpha-2s)^{-1} \right) \qquad (>0)
\label{eq:tstar}
\end{align}
and consider the corresponding $T=T_{\alpha,s,\rho}<M/(1-\alpha)$.
Then we immediately obtain
\[
\osc_{B_{4 \eps}(0)} \theta(T,\cdot) \geq \osc_{B_{4 \eps}(0)} \theta_{\alpha,s} (T,\cdot) \geq 2c \exp\left( - \gamma_{\alpha,s}(T) \right) \geq 2 c\exp\left( - c_{\alpha,s} (1-\alpha-2s)^{-1} \right) > \rho(8\eps)
\]
because $\theta,\theta_{\alpha,s}$ are smooth.  Thus $\rho$ is broken at time $T<M/(1-\alpha)$, with $M$ independet of $\rho$.
\end{proof}

\begin{remark}[\bf On bounds for  the critical case]
 In the critical case $\alpha+2s=1$, the above proof fails because if $T$ is such that $z_\alpha(T)=0$, then $\gamma_{\alpha,s}(T)=\infty$.  However,  \eqref{1.6} does yield
 \[ 
 \osc_{B_{4 z_\alpha(t)}(0)} \theta(t,\cdot) \geq  2c \exp\left(-\gamma_{\alpha,s}(t) \right) = 2cz_{\alpha}(t)^{c_{\alpha,s}M}
 \]
 for $t \in [0,T_{\alpha,s,\rho}]$. 
 Thus, the solution $\theta$ cannot have H\"older modulus of continuity $C^\beta$ for any $\beta > c_{\alpha,s}M$, so it becomes less regular as $M$ decreases. 
In particular, for $s=1/2$ and $\alpha=0$, we cannot have a H\"older modulus better than $C \|u\|_{L^\infty}^{-1}$, with a universal $C>0$. 
\end{remark}

The procedure described in this section may be summarized in one abstract theorem, which we shall use later in Section~\ref{sec:laplacian} to obtain blow-up in the case $s=1$, and in Section~\ref{sec:barely} for the case of a nonlocal operator $\EL$ which generalizes the fractional Laplacian, and is slightly supercritical. 

\begin{theorem}[\bf An abstract loss of regularity result] \label{thm:abstract}
Assume that  the drift-diffusion equation 
\begin{align} \label{1.11}
\theta_t+ u \cdot\nabla\theta +\EL \theta =0
\end{align}
on $(0,T)\times\RR^2$, with $u$ a continuous divergence-free vector field with  $u(t,Rx)=Ru(t,x)$ and $\EL$ a dissipative linear operator acting on $x$, satisfies Lemmas \ref{lemma:symmetries} and  \ref{lemma:comparison}.
Further assume that $u_1(t,0,x_2)=0$ and $\sgn(x_2)u_2(t,0,x_2)\le 0$ for $x_2\neq 0$. 
Let  $r\in(0,1)$, $\eta$ be a smooth bump function as above, define 
\[ \phi(x_1,x_2)=\eta(|(x_1,x_2-1)|r^{-1})-\eta(|(x_1,x_2+1)| r^{-1}),\] 
and let $z$ solve the ODE
\begin{equation} \label{1.20}
\dot{z}(t)=u_2(t,0,z(t)), \qquad z(0)=1
\end{equation}
on $(0,T)$.
Assume that for $\lambda>0$ and $\phi_\lambda (x)=\phi (x/\lambda)$ there exists $H(\lambda)>0$ such that
\begin{align} \label{1.10}
\EL\phi_\lambda \le H(\lambda) \phi_\lambda
\end{align}
on $\RR\times\RR^+$.
Finally, assume that for each $t\in(0,T)$, 
\begin{align} \label{1.13}
\left( u(t,x_1,x_2) - u_2(t,0,z(t))z(t)^{-1}(x_1,x_2) \right) \cdot (x_1,x_2-z(t)) \ge 0
\end{align}
holds on the disc $x_1^2+(x_2-z(t))^2\le rz(t)$. If $\theta$ is the solution of \eqref{1.11} with initial condition $\theta(0,\cdot)=c\phi$ for some $c>0$, then $\theta$ breaks at time $T$ any modulus of continuity $\rho$ with
\begin{align}\label{1.12}
\rho(2z(T)) < 2c\exp \left( - \int_0^{T} H(z(t)) \dd t \right).
\end{align}
\end{theorem}

Notice that if $z(T)=0$ (which can only happen if $u$ is not uniformly Lipschitz in $x$) and $\int_0^{T} H(z(t)) \dd t<\infty$, then $\theta$ breaks {\it any} modulus of continuity at time $T$, so we have finite time blow-up in $C(\RR^2)$.

It is clear that Theorem~\ref{thm:Holder} is a particular case of Theorem~\ref{thm:abstract}, with $\EL = (-\Delta)^{s}$ and $u=u_{\alpha,s,\rho}$ from \eqref{eq:ualpha}. Lemma~\ref{lemma:dissipation} and Corollary~\ref{cor:dissipation}  merely verify \eqref{1.10}, while \eqref{1.13} is verified in the proof of Lemma~\ref{lemma:sub:transport}. Finally, super-criticality $\alpha+2s<1$ guarantees condition \eqref{1.12} for any given modulus $\rho$ and an appropriate $T=T_{\alpha,s,\rho}$.

\begin{proof}[Proof of Theorem~\ref{thm:abstract}]
This follows as above because $\bar\phi(t,x)=\phi(x/z(t))$ is again an odd-in-$x_2$ sub-solution of $(\partial_t + u \cdot \nabla)\phi=0$ for $t\in(0,T)$ (see the proof of  Lemma~\ref{lemma:sub:transport} up to \eqref{1.8}), and so 
\begin{align}\label{1.14}
\bar\theta(t,x)= c\exp\left( -\int_0^{t} H(z(\tau)) \dd \tau \right) \phi\left( \frac{x}{z(t)}\right)
\end{align}
is an odd-in-$x_2$ sub-solution of \eqref{1.11} for $t\in(0,T)$ with $\bar\theta(0,\cdot)=c\phi$. Hence again
\[
\osc_{B_{z(T)}(0)} \theta(T,\cdot) \geq \osc_{B_{z(T)}(0)} \bar \theta (T,\cdot) \geq 2c\exp\left(  -\int_0^{T} H(z(\tau)) \dd \tau \right)  > \rho(2z(T)),
\]
so $\theta$ breaks $\rho$ at time $T$.
\end{proof}

\section{Finite time blow-up for non-autonomous drift-diffusion} \label{sec:laplacian}

In this section we prove Theorem \ref{t:intro3}. The argument in Section~\ref{sec:Holder} fails for $s=1$ when the drift velocity $u$ is time-independent. In fact, as we shall prove in Section~\ref{sec:regularity} that the result is simply not true in this case: the solution has a very weak modulus of continuity, globally in time, even if the time-independent drift is supercritical. In this section we prove that by allowing the drift velocity to depend on time, finite time loss of regularity can still be obtained in the  supercritical regime. We recall that for $s=1$ the critical space is $L^2(\RR^2)$.

\begin{proof}[Proof of Theorem~\ref{t:intro3}]
The proof is an application of Theorem~\ref{thm:abstract}.  We first let $q\in(p,2)$ and let
$z(t)$ be the solution of
\begin{align} 
\dot{z}(t) = - M^{-1} z(t)^{-2/q}, \qquad z(0)=1, \label{eq:z:dot}
\end{align}
with $M>0$ to be chosen later.
That is, 
\begin{align} 
z(t) = \left( 1- \frac{q+2}{q} M^{-1}t \right)^{q/q+2} \label{eq:z(t)}
\end{align}
for all $t \in [0,t_q]$, where $t_q = Mq/(q+2)$ is the time when $z$ reaches $0$. 

Next, we consider the vector  field
\begin{align} 
\bar u(x) =  \nabla^\perp \left(\frac{|x_1-x_2| - |x_1+x_2|}{2} \kappa \left(\frac {x_1}{x_2} \right) \mu(|x|) \right), \label{bar:u}
\end{align}
where $\kappa, \mu$ are smooth with $\kappa$ as in \eqref{eq:stream} and 
\[
\chi_{[1/2,2]} \le \mu \le \chi_{[1/3,3]}.
\]
That is, $\bar u$ is a smooth and compactly supported version of the vector field in Section~\ref{sec:Linfty}. 

Finally, we let $u$ be any smooth time-dependent divergence-free vector field with
\begin{align} 
u(t,x) = \frac{1}{Mz(t)^{2/q}}\; \bar u \left( \frac{x}{z(t)} \right) \label{u}
\end{align}
for $t\in[0,T_\rho]$ and reasonably behaved for $t\ge T_\rho$, where $T_\rho\in(0,t_q)$ will be chosen later.  We now pick $M>0$  large enough so that $\|\sup_t |u| \|_{L^p}\le 1$.  This is possible because
\[
\sup _{t\le T_\rho} |u(t,x)| \le \frac C{M|x|^{2/q}}
\] 
(for some constant $C$), and the left hand side vanishes for $|x|\ge 3$.

Then $u$ and $\EL=-\Delta$  satisfy all the hypotheses of Theorem~\ref{thm:abstract}, with $T=T_\rho$,  any $r\in(0,1/4)$, as well as for instance $\eta(|x|)= \exp(1-(1-|x|)^{-1})\chi_{[0,1]}(|x|)$ and $c>0$ such that then $\|c\phi\|_{C^2}\le 1$.
 In particular,  comparison and oddness in $x_2$ follow; 
\eqref{1.20} holds because $\bar u(0,1)=-(0,1)$ and \eqref{1.13} is just \eqref{1.30} (both for $t\le T_\rho$).

It remains to check that \eqref{1.10} holds with $H(\lambda)=C\lambda^{-2}$ for some $C>0$. By scaling, we only need to check it for $\lambda=1$. Since we are in the case $s=1$, which is local, only one of the terms in the definition of $\phi$ affects $\Delta \phi$ in the upper half plane:
\[ \Delta \phi(x) = \Delta \eta(r^{-1} |(x_1,x_2-1)|). \]
We start checking \eqref{1.10} for $x$ sufficiently close to $\partial B_r(0,1)$.  For $x$ with  $d=r^{-1}|(x_1,x_2-1)|$  close to $1$, we have 
\begin{align*}
-\Delta \phi(x) &= - r^{-2} \left( \eta''(d) - \frac{\eta'(d)}{d} \right) \\
 &= - r^{-2} \left( \frac 1 {(1-d)^4} - \frac 2 {(1-d)^3} - \frac 1 {d(1-d)^2} \right) \exp(1-(1-d)^{-1}) \leq 0.
\end{align*}
The last inequality holds for $d>1-\delta_0$ if $\delta_0$ is a small constant, since the first term $(1-d)^{-4}$ controls the other two. For $x$ such that $r^{-1} |(x_1,x_2-1)| \leq 1 - \delta_0$,
we have
\[
-\Delta \phi(x) \leq r^{-2} \| \Delta \eta \|_{L^\infty} \leq C \eta(1-\delta_0) \leq C \phi(x)
\]
by the monotonicity of $\eta$, if $C$ is chosen sufficiently large.

Since $H(z(t))= C z(t)^{-2}$, the right-hand side of \eqref{1.12} is strictly positive even if $T$ is replaced by $t_q$ (because $-2q/(q+2)>-1$).  Thus one only needs to choose $T_\rho$ sufficiently close to $t_q$ so that \eqref{1.12} holds.  The proof is finished.
\end{proof}

\section{Uniqueness of particle trajectories for divergence-free time-independent drifts in two dimensions} \label{sec:trajectories}

In this section we show that when the velocity of a transport equation is both {\em divergence-free} and {\em time-independent}, a seemingly unexpected regularity result for the associated Lagrangian trajectories holds.

Given a time-independent vector field $u$ which is divergence-free and does not vanish on a two-dimensional domain, we let $X(a,t)$ be the flow map associated to $u$, i.e. the solution of the ODE 
\begin{align} 
\frac{d X(a,t)}{dt} = u(X(a,t)), \qquad X(a,0)= a \label{eq:X}.
\end{align}
Although $u$ is not Lipschitz, since $u$ is divergence-free, we will show that the flow map $X(a,t)$ is unique, as long as $|u|$ stays away from $0$. More precisely, we have the following:

\begin{theorem}[\bf Uniqueness of particle trajectories]\label{lemma:unique:X}
 Let $u \colon \RR^2 \to \RR^2$ be a divergence-free continuous function such that $|u| \neq 0$ on a domain $\DD \subset \RR^2$. For any $a \in \DD$, the solution of \eqref{eq:X} is unique (forward and backward in time) as long as it stays in $\DD$.
\end{theorem}

\begin{proof}[Proof of Theorem~\ref{lemma:unique:X}]
Existence of solutions is standard because $u$ is continuous (Peano's theorem). We prove uniqueness by contradiction. Modulo changing the direction of time (i.e., replacing $u$ by $-u$), we can assume that there exist $a \neq b \in \DD$ and a minimal time $T>0$ such that $X(a,T)=X(b,T)$. Here $T$ is assumed to be small enough so that $X(a,t), X(b,t) \in \DD$ for all $t\in [0,T]$. Let $u_0 = u(X(a,T)) = u(X(b,T))$. Since $u$ is continuous, for small $t>0$ we have $X(a,T-t) = X(a,T) - u_0 t + o(t)$. Therefore, for any small $\eps>0$, we can find $t_1$ and $t_2$ small such that 
\[(X(a,T)-X(a,T-t_1))\cdot u_0 = (X(b,T)-X(b,T-t_2))\cdot u_0=\eps.\]
We define the paths
  
\begin{minipage}{2.5in}
\begin{align*}
& \gamma_1 = \{ X(a,T-t_1+s) \colon s \in [0,t_1]\} \\
& \gamma_2 = \{ X(b,T-s) \colon s \in [0,t_2]\} \\
& \gamma_3 = \{ s X(b,t_2) + (1-s) X(a,t_1) \colon s \in [0,1]\}.
\end{align*}
\end{minipage}
\begin{minipage}{2.5in}
\setlength{\unitlength}{1in} 
\begin{picture}(2.5625,2)
\put(0,0){\includegraphics[height=2in]{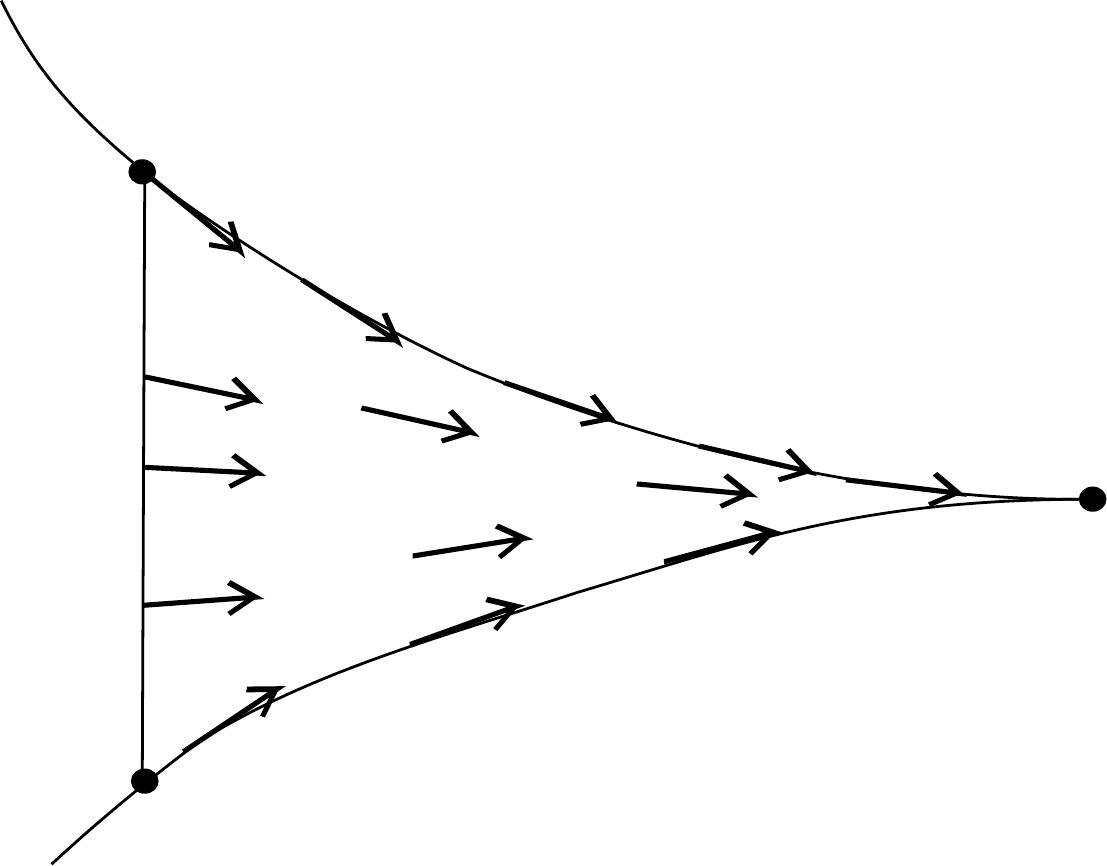}}
\put(1.12,1.19){$\gamma_1$}
\put(1.26,0.452){$\gamma_2$}
\put(0.132,0.868){$\gamma_3$}
\put(2.42,0.896){$X(a,T)$}
\put(2.42,0.696){$=X(b,T)$}
\put(0.271,1.68){$X(a,T-t_1)$}
\put(0.4,0.1){$X(b,T-t_2)$}
\put(0.7,0.8){$\Omega$}
\end{picture}
\end{minipage}

\medskip

Due to the minimality of $T$, it follows that $\gamma = \gamma_1 \cup \gamma_2 \cup \gamma_3$ is a piecewise smooth simple closed curve (let us orient it positively). Call $\Omega$ the region enclosed by $\gamma$ and let $n$ be the outward unit normal vector to $\partial \Omega$. Note that on $\gamma_1 \cup \gamma_2$ we have $u \cdot n  = 0$, and hence, applying Green's theorem in $\Omega$ yields
\begin{align*}
 0 = \int_{\Omega} \nabla \cdot u = \int_{\gamma} u \cdot n = \int_{\gamma_3} u \cdot n < 0.
\end{align*}
In the last inequality we use that on $\gamma_3$, $n = -u_0/|u_0|$ and, since $\eps$ is small and $u$ is continuous then $u$ is close to $u_0$ along $\gamma_3$.
\end{proof}

\begin{remark}
Note that Theorem \ref{lemma:unique:X} immediately implies that the transport equation
\[ \partial_t \theta + u \cdot \nabla \theta = 0,\]
has a continuous solution $\theta$ if the initial data $\theta_0$ is continuous and the drift $u$ is time-independent, continuous, divergence-free, and non-vanishing.
\end{remark}

\section{A modulus of continuity for autonomous drift-diffusion}\label{sec:regularity}

In this section we prove Theorem \ref{t:intro4}.  That is, we consider the equation 
\begin{align} 
&\partial_t \theta + u(x) \cdot \grad \theta - \Delta \theta = 0 \label{e:autonomous}\\
&\theta(0,\cdot) = \theta_0 \label{e:auto:IC}
\end{align}
in the full space $\RR^+\times \RR^2 $, with a divergence-free vector field $u\in L^1_{\rm loc}$ . It is important here that $u$ is time-independent, and it will be used in the proof that $\partial_t \theta$ satisfies the same equation. 

We say that such $u$ is divergence-free (in the distributional sense) if $\int_{\RR^2} u(x)\cdot\nabla\phi(x)dx=0$ for each $\phi\in C^\infty_0(\RR^2)$, and $\theta\in L^\infty (\RR^+\times \RR^2)$ is a distributional solution of \eqref{e:autonomous}--\eqref{e:auto:IC} if
\begin{equation} \label{6.6}
\int_{\RR^+\times \RR^2} \theta(t,x) (\phi_t(t,x) + u(x)\cdot \nabla\phi(t,x) + \Delta\phi(t,x)) dtdx = \int_{\RR^+} \theta_0(x)\phi(0,x) dx
\end{equation}
for each $\phi\in C^\infty_0(\RR^3)$.
Note that it is enough to prove the a priori estimate in Theorem \ref{t:intro4} for smooth $u$ and smooth classical solutions. Indeed, the result of the theorem then follows by a standard approximation of the drift $u$ in the $L^1_{\rm loc}$ norm by smooth divergence-free drifts $u_n$, for instance, via a mollification. 
The uniform a priori modulus of continuity for the associated smooth solutions $\theta_n$ (spatial from \eqref{e:MOC} and temporal from the last bound in Lemma \ref{l:linfty-decay}) and the maximum principle for $\theta_n$ show that  $\theta_n$ converge (along a subsequence) locally uniformly in $\RR^+\times\RR^2$ to some  $\theta\in L^\infty (\RR^+\times \RR^2)$ satisfying the same modulus of continuity.  Despite this modulus blowing up as $t\to 0$, since all the $(\theta_n,u_n)$ are uniformly bounded in  $L^\infty\times L^1_{\rm loc}$ and $\phi$ is compactly supported, it follows that the right-hand sides of \eqref{6.6} for $(\theta_n,u_n)$ converge to the right-hand side of \eqref{6.6} for $(\theta,u)$.  Thus $\theta$ is a distributional solution of \eqref{e:autonomous}--\eqref{e:auto:IC}.
Note that we do not claim that the weak solution of the problem with $u\in L^1_{\rm loc}$ but not smooth is unique, and there might be other distributional solutions, which could  be discontinuous if $u$ is not in $L^2$.  If $u$ is not at least in $L^1_{\rm loc}$, the definition of divergence-free fields and distributional solutions is questionable.

In the rest of this section we will assume that $u$ is smooth and hence so is $\theta$.  The idea of the proof is as follows. It was noted in \cite{SSSZ12} that any function in $\RR^2$ which satisfies the maximum principle on every disk and belongs to $H^1$ has to be continuous, with a logarithmic modulus of continuity depending on its $H^1$ norm only. The $H^1$ estimate, which comes from the energy inequality, is a critical quantity in terms of continuity of the solution, since any estimate in $H^{1+\eps}(\RR^2)$ would imply a H\"older modulus of continuity from classical Sobolev embeddings. It is well known that a function in $H^1(\RR^2)$ is not necessarily continuous. However, if the function solves some elliptic PDE, the maximum principle can be used to bridge that gap and obtain a logarithmic modulus of continuity. 

For the case of parabolic equations, we have two difficulties in carrying out this scheme. The first one is that the energy inequality only gives us an estimate in $L^2(H^1)$, which is far from implying continuity. However, by differentiating the equation in time (and using that $u$ depends only on $x$) we can get that the solution is actually bounded in $C(H^1)$ and Lipschitz in time. The other difficulty is more severe, and it is that the parabolic maximum principle is quite different from the elliptic one. For example, if we fix any function $f \in H^1(\RR^2)$ and extend it as constant in time, this function belongs to $C(H^1)$, is certainly Lipschitz in time, and satisfies the parabolic maximum principle without necessarily being continuous. In order to overcome this second difficulty we prove that at each fixed time $t$, the solution approximately satisfies the elliptic maximum principle (Lemma~\ref{l:elliptic-max}). In order to obtain this approximate maximum principle at each $t$, we need to use the equation again, applying a non-trivial result due to J. Nash (see Theorem~\ref{t:nash} and Corollary~\ref{c:localDG}).


\subsection{Energy estimates}

We start by discussing some estimates on the fundamental solution $G$ of \eqref{e:autonomous}. Recall that $G$ is by definition the solution of the equation
\[ \partial_t G(t,x,y) + u(x)\cdot \grad_x G(t,x,y) - \lap_x G(t,x,y) = 0,\]
with $\lim_{t \to 0} G(t,\cdot,y) = \delta_y$. With respect to $y$, it solves the dual equation
\[ \partial_t G(t,x,y) - u(y)\cdot \grad_y G(t,x,y) - \lap_y G(t,x,y) = 0,\]
with $\lim_{t \to 0} G(t,x,\cdot) = \delta_x$.

The fundamental solution is used to compute the solutions of \eqref{e:autonomous} via the formula
\[ \theta(t,x) = \int G(t,x,y) \theta_0(y) \dd y.\]
It is a simple consequence of the maximum principle that $G$ is nonnegative and $\int G(t,x,y) \dd y = 1$ for all $x$ and $t$. The following is a result by J.~Nash on the size of the fundamental solution. The remarkable fact of this estimate is that it is independent of $u$, as long as it is divergence-free.

\begin{theorem} [\bf J.~Nash~\cite{Nash58}] \label{t:nash}
If $u$ is divergence-free in $\RR^2$, then the fundamental solution $G(t,x,y)$ of \eqref{e:autonomous} satisfies the pointwise bound
\[ 0 \leq G(t,x,y) \leq C/t ,\]
where $C$ is a universal constant (independent of $u$).
\end{theorem}

The theorem was not stated explicitly with these assumptions by Nash, therefore we include the proof below. It is a straight-forward adaptation of the proof in \cite{Nash58} to our setting.

\begin{proof}
We compute the evolution of the $L^2$ norm of $G$. For any fixed $y$, let $E(t) =\|G(t,\cdot,y)\|_{L^2}^2$ (which is finite for $t>0$).
We have that
\begin{align*}
E'(t) &= 2\int_{\RR^2} G(t,x,y) (-u \cdot \nabla_x G + \lap_x G) \dd x
= -2\int_{\RR^2} |\nabla_x G(t,x,y)|^2 \dd x 
\leq -C^{-1} E(t)^2.
\end{align*}
for some universal $C>0$.  In the last inequality we used the interpolation inequality ({\it Nash's inequality})
\[ c \|f\|_{L^2}^{1+2/d} \leq \|f\|_{\dot H^1} \|f\|_{L^1}^{2/d},\]
together with the fact that $\|G(t,\cdot,y)\|_{L^1}=1$ for all $t$ and $y$.

The ODE $E'(t) \leq -E(t)^2/C$ implies that $E(t) \leq C/t$. Thus, we have shown a universal bound $\|G(t,\cdot,y)\|_{L^2}^2 \leq C/t$. Since the estimate does not depend on $u$ (in particular it is the same if we replace $u$ by $-u$) we also have $\|G(t,x, \cdot)\|_{L^2}^2 \leq C/t$.

The pointwise estimate now follows from the energy estimate via the semigroup property of $G$:
\begin{align*}
G(t,x,y) = \int_{\RR^2} G(t/2,x,z) G(t/2,z,y) \dd z  
\leq \|G(t/2,x,\cdot)\|_{L^2} \|G(t/2,\cdot,y)\|_{L^2} \leq C/t.
\end{align*}
The proof is finished.
\end{proof}


\begin{lemma} \label{l:linfty-decay}
Let $\theta$ be a smooth solution of \eqref{e:autonomous} in $\RR_+\times \RR^2$. Then 
\begin{align*}
\| \theta(t,\cdot) \|_{L^1(\RR^2)} &\leq \| \theta_0 \|_{L^1(\RR^2)} \\
\| \theta(t,\cdot) \|_{L^\infty(\RR^2)} &\leq \| \theta_0 \|_{L^\infty(\RR^2)} \\
\| \partial_t \theta(t,\cdot) \|_{L^1(\RR^2)} &\leq  {\left( 1+ \| u \|_{L^1_{\rm loc}(\RR^2)}   \right)  \| \theta_0 \|_{W^{4,1}(\RR^2)}} \\
\|\partial_t \theta(t,\cdot) \|_{L^\infty(\RR^2)} &\leq { \| \theta_0 \|_{C^2(\RR^2)}+ C {t}^{-1} \|u \|_{L^1_{\rm loc}(\RR^2)} \| \theta_0 \|_{W^{4,1}(\RR^2)}  }
\end{align*}
for some universal $C>0$.
\end{lemma}

\begin{proof}
The first two estimates are just the maximum principle for parabolic equations, which also holds for all $L^p$ norms with $p<\infty$ because $u$ is divergence-free.

Since the drift is time independent,  $\partial_t \theta$ solves the same equation
\[ \partial_t ( \partial_t \theta) + u \cdot \grad (\partial_t  \theta) - \lap (\partial_t \theta) = 0.\]
The third estimate follows now directly from  $\partial_t \theta(0,\cdot) = -u \cdot \grad \theta_0 + \lap \theta_0$ and the maximum principle, using also
\begin{align} 
\|u\cdot \nabla\theta_0\|_{L^1} 
&\le \|u\|_{L^1_{\rm loc}} \sum_{n\in\ZZ^2} \|\nabla \theta_0\|_{L^\infty(n+[0,1]^2)} \notag\\
&{\le c \|u\|_{L^1_{\rm loc}}\left(  \|\nabla \theta_{0}\|_{L^{1}} + \|\nabla^{4} \theta_0\|_{L^1} \right) \le c \|u\|_{L^1_{\rm loc}}  \|\theta_0\|_{W^{4,1}}.}
\label{6.1}
\end{align} 
{Note that if instead of $u \in L^{1}_{\rm loc}$ we assume that $u \in L^{1}$, then one may lower the regularity assumption on the initial data to $\theta_{0} \in C^{2} \cap W^{2,1}$.}
Finally, if $G$ is the fundamental solution, then
\begin{align*}
\partial_t \theta(t,x) &= \int G(t,x,y) \partial_t \theta(0,y) \dd y\\
&= \int G(t,x,y) (-u(y) \cdot \grad \theta_0(y) + \lap \theta_0(y)) \dd y \\
&\leq \| \lap \theta_0 \|_{L^\infty} + \int G(t,x,y) \left| u(y) \cdot \grad \theta_0(y)\right| \dd y \\
&\leq {\|\theta_0\|_{C^2} +\| G(t,\cdot,\cdot) \|_{L^\infty} \|u \cdot \nabla \theta_{0}\|_{L^1} } \\
& \leq { \|\theta_0\|_{C^2}+  C t^{-1} \|u\|_{L^1_{\rm loc}} \| \theta_0 \|_{W^{4,1}(\RR^2)}  },
\end{align*}
where in the last inequality we applied Theorem~\ref{t:nash}, and in the one before we used \eqref{6.1}. 
\end{proof}

By Lemma \ref{l:linfty-decay}, both $\theta$ and $\partial_t \theta$ are bounded in $L^1 \cap L^\infty$ for positive time, and in particular, they are bounded in $L^2$ for every $t>0$. The following lemma can then be applied to obtain estimates in space-time norms.

\begin{lemma} \label{l:energy-ineq}
Let $\theta$ be a smooth solution of \eqref{e:autonomous} in $\RR_+\times \RR^2$.  Then
\begin{align*}
\| \theta \|_{L^2((t_0,\infty),\dot H^1)} &\leq \|\theta(t_0,\cdot)\|_{L^2} \\
\| \partial_t \theta \|_{L^2((t_0,\infty),\dot H^1)} &\leq \|\partial_t \theta(t_0,\cdot)\|_{L^2} \\
\| \theta \|_{L^\infty((t_0,\infty),\dot H^1)} &\leq C \left( \|\theta(t_0,\cdot)\|_{L^2} + \|\partial_t \theta(t_0,\cdot)\|_{L^2} \right)
\end{align*}
for some universal $C>0$.
\end{lemma}

\begin{proof}
The first two estimates are the classical energy estimates for the equations for $\theta$ and $\partial_t \theta$ respectively. Note that since $u$ is divergence-free, the drift term has no effect on the energy estimate. The first and second inequalities can be interpreted as that $\theta \in H^1((0,\infty),\dot H^1)$. The third inequality follows therefore from the one dimensional Sobolev embedding $H^1 \hookrightarrow C^{1/2}$.
\end{proof}

\begin{remark}
The bounds on $u$ in $L^1_{\rm loc}$ and $\theta_0$ in ${C^2 \cap W^{4,1}}$ {(alternatively, on $u\in L^{1}$ and $\theta_{0} \in C^2 \cap W^{2,1}$)} are only used to obtain estimates on $\partial_t \theta$ in $L^{\infty}((t_{0},\infty);L^{\infty}(\RR^{2}))$ and on $\nabla \theta$ in $L^{\infty}((t_{0},\infty);L^{2}(\RR^{2}))$ for any $t_{0}>0$.  These estimates can also be obtained in terms of, for instance, 
\begin{align*}
\| \theta_0\|_{L^2} = A_{1} < \infty \qquad \mbox{and} \qquad
\|u\cdot \nabla \theta_{0}  - \Delta \theta_{0}\|_{L^2} = A_{2} < \infty.
\end{align*}
Indeed, from $\partial_{t} \theta(t,x) = \int G(t,x,y) \partial_{t} \theta(0,y) \dd y$ and  $\|G(t,x,\cdot)\|_{L^{2}} \leq C t^{-1/2}$, it follows that  $\| \partial_{t} \theta(t,\cdot)\|_{L^{\infty}} \leq C t^{-1/2}A_{2}$. Also, the maximum principle for the equation obeyed by $\partial_{t} \theta$ implies $\| \partial_{t} \theta(t,\cdot)\|_{L^{2}} \leq A_{2}$, and from the energy inequality we get an estimate on $\partial_{t} \theta$ in $L^{2}((t_{0},\infty); \dot{H}^{1}(\RR^{2}))$. The bounds in Lemmas~\ref{l:linfty-decay} and \ref{l:energy-ineq} follow, and they depend on  the (smooth) drift $u$ only implicitly through $A_{2}$.
\end{remark}

\subsection{A maximum principle on time slices}

In this sub-section we denote $B_r=B_r(0)\subseteq\RR^2$.

\begin{lemma} \label{t:localDG}
There exists $h>0$ such that for any smooth divergence-free $u$ on $B_1$, the solution of\begin{align*}
\partial_t \phi + u \cdot \grad \phi - \lap \phi &= 0 \qquad \text{in } (-h,0) \times B_1  \\ 
\phi(-h,\cdot ) &= 1 \qquad \text{in } B_1\\
\phi &= 0 \qquad \text{in } (-h,0) \times  \partial B_1
\end{align*}
satisfies $\sup_{B_1} \phi(0,\cdot) \leq 1/2$.
\end{lemma}

\begin{proof}
The function $\phi$ is nonnegative, so if we extend it by zero outside of $B_1$ for all times $t$  (and extend $u$ in any way we like), we obtain a sub-solution of the equation \eqref{e:autonomous} in space $(-h,0) \times \RR^2$. Therefore, by the comparison principle, it has to stay below the actual solution with the same initial values at $\{-h\} \times \RR^2$:
\[ \phi(0,x) \leq \int G(h,x,y) \chi_{B_1}(y) \dd y \leq  C h^{-1},\]
where we applied Theorem \ref{t:nash} in the last inequality. So $h=2C$ works.
\end{proof}

\begin{corollary} \label{c:localDG}
With $h$ from Lemma \ref{t:localDG}, we have that for any $a>b$, the solution   of
\begin{align*}
\partial_t \phi + u \cdot \grad \phi - \lap \phi &= 0 \qquad \text{in }  (-h r^2,0) \times B_r \\ 
\phi(-h r^2,\cdot) &\leq a \qquad \text{in } B_r \\
\phi &\leq b \qquad \text{in } (-hr^2,0) \times \partial B_r 
\end{align*}
satisfies $\sup_{B_r} \phi(0,\cdot) \leq (a+b)/2$.
\end{corollary}

\begin{proof}
Compare  $\phi$ with $b + (a-b) \tilde \phi(r^2 t,rx) /2$, where $\tilde \phi$ is the function from Lemma \ref{t:localDG}. 
\end{proof}

We recall the maximum principle for parabolic equations:
\begin{proposition}\label{p:maxprin}
For any parabolic cylinder $Q = (-h,0) \times B_r $, we have that 
\[ \sup_{Q} \theta = \sup_{\partial_p Q} \theta,\] 
where $\partial_p Q$ denotes the parabolic boundary
\[ \partial_p Q =  \left( \{-h\} \times B_r\right) \cup \left( (-h,0) \times \partial B_r \right).\]
\end{proposition}

The following lemma gives a relation on each time slice which approximates the maximum principle for elliptic equations.

\begin{lemma}[\bf A maximum principle on time slices] \label{l:elliptic-max}
Let $h$ be the constant from Lemma \ref{t:localDG} and assume that  $\theta$ is a solution of \eqref{e:autonomous} in $[-T,0] \times B_{\sqrt {T/h}}$, with $\partial_t \theta$ bounded. Then for any  $R<\sqrt{T/h}$ we have
\begin{align*}
\sup_{B_R} \theta(0,\cdot) \le \sup_{\partial B_R} \theta(0,\cdot) + 2\|\partial_t \theta\|_{L^\infty} hR^2, \\
\inf_{B_R} \theta(0,\cdot) \ge \inf_{\partial B_R} \theta(0,\cdot) - 2\|\partial_t \theta\|_{L^\infty} hR^2.  
\end{align*}
\end{lemma}

\begin{proof}
We only prove the first claim, the second  follows by considering $-\theta$ instead of $\theta$.
Let $a = \sup_{B_R} \theta(-hR^2,\cdot)$ and $b = \sup_{\partial B_R} \theta(0,\cdot) + \|\theta_t\|_{L^\infty} hR^2$.  Then obviously $b\ge \max_{(-hR^2,0) \times \partial B_R } \theta$. If $a\le b$,  the maximum principle immediately yields $\sup_{B_R} \theta(0,\cdot) \le b$ so let us assume $a>b$.

On one hand, from Corollary \ref{c:localDG} we have that 
\[ \sup_{B_R} \theta(0,\cdot) \leq \frac{a+b}{2} \leq \frac{a + \sup_{\partial B_R} \theta(0,\cdot) + \|\theta_t\|_{L^\infty} hR^2} 2.\]
On the other hand, from  boundedness of $\partial_t \theta$ we also have
\[ a \leq \sup_{B_R} \theta(0,\cdot) + \|\theta_t\|_{L^\infty} hR^2.\]
Combining the last two inequalities, we obtain the first claim.
\end{proof}

\subsection{A local modulus of continuity and the proof of Theorem~\ref{t:intro4}}

From Lemmas~\ref{l:linfty-decay} and~\ref{l:energy-ineq}, we have that the hypotheses of Theorem~\ref{t:intro4} imply that $\|\theta\|_{L^\infty}$, $\|\partial_t \theta(t,\cdot)\|_{L^\infty}$, and $\|\theta(t,\cdot)\|_{H^1}$ are bounded on any time-interval $[t_0,\infty)$ with $t_0>0$. The proof will be finished by using the following local result.  It applies Lemma \ref{l:elliptic-max} on each time slice, so that we can now more or less follow the idea from the proof in \cite{SSSZ12}.  

\begin{theorem}[\bf Local continuity with supercritical drift] \label{t:modulus}
Let  $u$  be a smooth time-independent divergence-free vector field on $B_1(0)\subseteq \RR^2$ and assume that the solution $\theta$ of \eqref{e:autonomous} on $(t_0,t_1)\times B_1(0)$ satisfies $ \|\partial_t \theta(t,\cdot)\|_{L^\infty(B_1(0))}\le D$, $\|\nabla \theta(t,\cdot)\|_{L^2(B_1(0))}\le E$, and $\|\theta(t,\cdot)\|_{L^\infty(B_1(0))} \le F$, uniformly in $t\in (t_0,t_1)$.  Then $\theta(t,\cdot)$ restricted to $B_{1/2}(0)$  satisfies a modulus of continuity given by
\begin{align}
\rho_t(r) = \frac {C (D+ E + F \sqrt{\log_-(t-t_0)}\,)} {\sqrt{ - \log r}}
\label{6.3}
\end{align}
for all $r \in (0,1/2)$ and  $t \in (t_0,t_1)$, and some universal $C>0$ (with $\log_- s=\max\{-\log s, 0  \}$).   
\end{theorem}

\begin{proof}[Proof of Theorem~\ref{t:modulus}]
Fix $x\in B_{1/2}(0)$ and $t>t_0$, and let $T=\min \{ t-t_0, h/4\}$, with $h$ from Lemma \ref{t:localDG}.  Let also 
$c=2hD$. 
We now apply Lemma \ref{l:elliptic-max} to $\theta$ shifted by $(t,x)$ and get for any $s<\sqrt{T/h}\leq 1/2$,
\begin{align*}
\sup_{\partial B_s(x)} \theta(t,\cdot) \geq \sup_{ B_s(x)} \theta(t,\cdot) - c s^2, \\
\inf_{\partial B_s(x)} \theta(t,\cdot) \leq \inf_{ B_s(x)} \theta(t,\cdot) + c s^2.
\end{align*}
In particular
\begin{equation}~\label{mon}
\osc_{\partial B_s(x)} \theta(t,\cdot) \geq \osc_{B_s(x)} \theta(t,\cdot) - 2c s^2.
\end{equation}

We first prove that
\begin{equation} \label{6.2}
 (\osc_{B_r} \theta=) \quad\osc_{B_r(x)} \theta(t,\cdot) < \max\left\{4c r , \frac{4\pi E}{\sqrt{ - \log r}} \right\} 
\end{equation}
holds for any $r< T/h \,(\le 1/4)$, after which we proceed to all $r<1/2$. 
To prove \eqref{6.2} we just need to consider the case when $\osc_{ B_r} \theta \ge 4cr $.  Let $R=\sqrt{r}<\sqrt{T/h} $
and estimate
\begin{align*}
E^2 \geq   \int\limits_{B_R \setminus B_r} |\grad \theta|^2 \dd x &= \int\limits_r^R \int\limits_{\partial B_s} |\grad \theta|^2 \dd \sigma \dd s \geq \int\limits_r^R \int\limits_{\partial B_s} |\theta_\sigma|^2 \dd \sigma \dd s,
\end{align*}
since $|\grad \theta|^2 = \theta_\sigma^2 + \theta_\nu^2$  where $\theta_\sigma$ is the tangential derivative and $\theta_\nu$ is the normal one. We rewrite the integral on the right using polar coordinates $(s \hat \sigma = \sigma)$, and use  Cauchy-Schwartz
 to obtain
\begin{align*}
E^2 &\geq \int\limits_r^R \frac 1 s \int\limits_{\partial B_1} |\theta_{\hat \sigma}(s \hat \sigma)|^2 \dd \hat \sigma \dd s
\geq \int\limits_r^R \frac 1{\pi^2s}  (\osc_{\partial B_s} \theta)^2 \dd s.
\end{align*}
Applying estimate \eqref{mon},  and noticing that $2cs^2\le 2c R^{2} =  2c r \leq \tfrac 12\osc_{ B_r} \theta $, we get
\begin{align*}
{\pi^2}{E^2} \geq \int\limits_r^R \frac 1 s (\osc_{ B_r} \theta - 2cs^2)^2 \dd s \ge (\log R-\log r) \left( \frac 12\osc_{ B_r} \theta \right)^2
\geq \frac{-\log r}{8} (\osc_{ B_r} \theta)^2.
\end{align*}
 Thus \eqref{6.2} holds for $r \in (0,T/h)$.  On the other hand, we have that
\begin{align} 
\osc_{B_r} \theta \leq 2 \|\theta\|_{L^\infty(B_r)} \leq 2 F \label{e:triangle:ineq}
\end{align}
holds for $r\in[T/h,  1/2)$.  Thus, in order to prove \eqref{6.3}, we only need to combine \eqref{6.2} and \eqref{e:triangle:ineq} with the fact  that $r\le (-\log r)^{-1/2}$ for $r< 1/2$.
\end{proof}

The proof of Theorem~\ref{t:intro4} now follows by fixing any $x\in\RR^2$ (without loss take $x=0$) and $t>0$, then letting $t_0= t/2$ (with $h$ from Lemma \ref{t:localDG}) and $t_1=\infty$.  Lemmas~\ref{l:linfty-decay} and~\ref{l:energy-ineq} imply that 
\begin{align*}
&D \leq { C (1 +   t^{-1} \| u\|_{L^1_{\rm loc}}) \|\theta_0\|_{C^2\cap W^{4,1}}  }\\
&E \leq { C (1 + \|u\|_{L^{1}_{\rm loc}} + t^{-1} \| u\|_{L^1_{\rm loc}} ) \|\theta_0\|_{C^2\cap W^{4,1}} }\\
&F \leq \| \theta_0\|_{L^\infty}
\end{align*}
which combined with \eqref{6.3} implies the estimate \eqref{e:MOC} of Theorem~\ref{t:intro4}. Lastly, by Lemma \ref{l:linfty-decay},
$\theta$ is uniformly Lipschitz in time on $[t_0,\infty) \times \RR^2$, for any $t_0>0$, and hence continuous on $\RR^+\times\RR^2$.

\section{Loss of regularity in the slightly supercritical case} \label{sec:barely}
Motivated by \cite{DKSV12}, in this section we address the regularity of solutions to the drift-diffusion equation
\begin{align} 
\partial_{t} \theta + u \cdot \nabla \theta + \EL \theta = 0 \label{eq:PDE}
\end{align}
with smooth initial condition $\theta(0,x) = \theta_{0}(x)$ and  a bounded divergence-free drift  $u$. Here $\EL$ is a nonlocal dissipative operator which is slightly less smoothing than $(-\Delta)^{1/2}$.  More precisely, let $m \colon \RR^{2}\setminus\{0\} \to (0,\infty)$ be a smooth radially symmetric, radially decreasing function, that is singular at the origin, decays at infinity, and satisfies the below properties:
\begin{align} 
& \int_0^\infty \frac{m(r)}{1+r} \dd r  < + \infty \label{eq:m:cond:1}\\
& r m(r) \mbox{ is non-decreasing for } r \in (0,1) \label{eq:m:cond:3}.
\end{align}We abuse notation and write $m(y) = m(|y|)$ for $y \in \RR^2 \setminus \{0\}$. Condition \eqref{eq:m:cond:3} can be relaxed, to $r^\beta m(r)$ is non-decreasing on $(0,1)$ for some $\beta < 2$. Associated to this function $m$ we define the nonlocal operator 
\begin{align}\label{eq:EL}
\EL \theta(x) = P.V. \int_{\RR^{2}} \left( \theta(x) - \theta(x+y) \right) \frac{m(y)}{|y|^{2}} \dd y 
\end{align}
for all smooth functions $\theta$. Condition \eqref{eq:m:cond:1} is the only essential one, and shows that we may take $m(r) \approx r^{-s}$ for any $s\in (0,1/2)$, but also $m(r) \approx r^{-1} ( \log(2 + 1/r))^{-\beta}$ for any $\beta > 1$. We informally say that $\EL$ is less smoothing than $(-\Delta)^{1/2}$ by {at least a logarithm}, or that $\EL$ is slightly supercritical with respect to the scaling induced by $L^\infty$ drift (for which $(-\Delta)^{1/2}$ is critical).

\begin{proof}[Proof of Theorem~\ref{t:intro5}]  The proof  consists of applying Theorem~\ref{thm:abstract} to the drift-diffusion equation \eqref{eq:PDE}. Since we wish to consider drift in $L^\infty(\RR^2)$ it is natural to take the same drift as in Section~\ref{sec:Linfty}, but which we smoothen as done in Section~\ref{sec:Holder}. The main difficulty arrises in proving condition \eqref{1.10} in Theorem~\ref{thm:abstract}. The issue is that by considering $\EL$, we have lost the homogeneity of the associated kernel. Most of the analysis is devoted to finding a constant $H(a)$ such that \eqref{1.10} holds for all $a>0$.

Since we work with bounded drift, we set $u_\rho(x) = u_{0,\rho}(x)$, which is defined by letting $\alpha=0$ (and $r_\alpha = 1/4$) in \eqref{eq:ualpha}. 
In particular, $u$ is divergence-free, has $L^\infty$ norm independent of $\rho$ (which we normalize to be less than $1$), is globally smooth and vanishes in a ball $B_{\eps_{\rho,m}}(0)$, where $\eps_{\rho,m}$ will be chosen later on depending on $\rho$ and $m$.
In addition, for $x  \in \CC_\rho = (B_{100}(0) \setminus B_{\eps_{\rho,m}}(0) ) \cap (\CC_0 \cup \CC_0')$  
we have the explicit formula
\begin{align}
u_{\rho,m}(x_1,x_2) =  \sign(x_2) (0,-1)  \label{eq:u}
\end{align}
where $\CC_0$ is the cone centered at the origin which is tangent to $B_{1/2}(0,1)$, and $\CC_0'$ is the reflexion of $\CC_0$ about the origin. In particular 
\begin{align} \label{1.2:barely}
u(0,x_{2})=(0,-1)
\end{align}
for all $x_{2} \in (\eps_{\rho,m},100)$. Thus $u_{\rho,m}$ satisfies the conditions of Theorem~\ref{thm:abstract}.

As in Section~\ref{sec:Holder}, since the above drift obeys the symmetries of the problem and is smooth, we have that \eqref{eq:PDE} has a comparison principle on the upper half-plane, and solutions are odd in $x_2$. 

Let $\eta\in C^\infty(\RR)$ be a radially non-increasing, cutoff function supported on $|x|\leq 1$, with  $\eta(x)>0$ if $|x|<1$, and $\eta(x) =1$ if $|x|\leq 1/2$. We set
\begin{align} 
\phi(x_1,x_2) = \eta(4 |(x_1,x_2-1)|) - \eta (4 |(x_1,x_2+1)|). \label{eq:phi:barely}
\end{align}
Without loss of generality, we may assume that $\theta_0 \geq \phi(x)$ in the upper half-plane $\{x_2>0\}$. As outlined in Theorem~\ref{thm:abstract}, we define $z(t)$ as the solution of
\[
\dot z(t) = u_2(0,z(t))= (0,-1), \qquad z(0)=1,
\]
which is given explicitly as 
\begin{align} 
z(t) = 1- t \label{eq:z}
\end{align}
for all $t\in [0,t_{\rho,m}]$. Here $t_{\rho,m} = 1 - \eps_{\rho,m}$ is the time it takes $z(t)$ to reach the value $\eps_{\rho,m}$ (we need to work in this time interval in order to use the explicit formula \eqref{1.2:barely}). Since the vector field considered here is the same as for $\alpha =0$ in Section~\ref{sec:Holder}, condition \eqref{1.13} of Theorem~\ref{thm:abstract} automatically holds by \eqref{1.8} in the proof of Lemma~\ref{lemma:sub:transport}.

For any $a>0$, we define $\phi_a(x) = \phi(x/a)$, where $\phi$ is defined in \eqref{eq:phi:barely}. We now need to quantify the effect of $\EL$ on $\phi_a$. Verifying that condition \eqref{1.10} of Theorem~\ref{thm:abstract} holds for some constant $H(a)$ is more delicate than in Section~\ref{sec:Holder}, so we state this in a lemma, which we shall prove at the end of this section.
\begin{lemma}[\bf Effect of dissipation]\label{lemma:diss}
For any $a\in [0,1]$ we have that
\begin{align} 
\EL \phi_a(x) \leq H(a) \phi_a(x) \label{eq:eigen}
\end{align}
holds in the upper half plane $\{x_2>0\}$ where 
\begin{align} 
H(a)= c_0 \left( \frac{1}{a^2} \int_0^{a} r m(r) \dd r +  \int_{a}^1 \frac{m(r)}{r} \dd r +1 \right) \label{eq:gamma}
\end{align}
for some positive constant $c_0$ which may depend on $\eta$ and $m$, but not on $a$.
\end{lemma}

We conclude the proof of Theorem~\ref{t:intro5}, and then return to prove Lemma~\ref{lemma:diss}. It is left to verify condition \eqref{1.12} in Theorem~\ref{thm:abstract}. For this purpose we first prove that
\begin{align} 
\int_{0}^{t_{\rho,m}} H(z(t)) \dd t \leq \bar C < \infty, \label{eq:barely:tocheck}
\end{align}
for some constant $\bar C$ that is independent of $\eps_{\rho,m}$ (and hence independent of $\rho$). 
Note that once \eqref{eq:barely:tocheck} is proven,  we have that at time $T = t_{\rho,m} = 1 - \eps_{\rho,m}$
\[
\rho(2z(T)) = \rho(2 \eps_{\rho,m}) \leq \exp(-\bar C) \leq \exp\left( -\int_0^T H(z(t)) \dd t\right)
\]
by choosing $\eps_{\rho,m}$ sufficiently small, thereby proving condition \eqref{1.12} in Theorem~\ref{thm:abstract}. 

We now prove \eqref{eq:barely:tocheck}. By \eqref{eq:z} and \eqref{eq:gamma}, since $c_0$ does not depend on $\eps_{\rho,m}$, and since $t_{\rho,m} \leq 1$, it is sufficient to estimate
\begin{align*} 
& \int_0^{t_{\rho,m}} \left(\frac{1}{z(t)^2} \int_0^{z(t)} r m(r) \dd r +  \int_{z(t)}^1 \frac{m(r)}{r} \dd r \right) \dd t  \\
&\qquad \leq \int_0^1 \left(\frac{1}{(1-t)^2} \int_0^{1-t} r m(r) \dd r +  \int_{1-t}^1 \frac{m(r)}{r} \dd r \right) \dd t.
\end{align*}
First, using Fubini we have
\[ 
\int_0^{1} \int_{1-t}^1 \frac{m(r)}{r} \dd r \dd t = \int_0^1 \frac{m(r)}{r} \int_{1-r}^1 1 \dd t \dd r = \int_0^1 m(r) \dd r < \infty.
\]
Second, cf.~\ref{eq:m:cond:3} we use that  $sm(s)$ is non-decreasing on $(0,1)$, and obtain
\[
\int_0^1 \frac{1}{(1-t)^2} \int_0^{1-t} r m(r) \dd r \dd t \leq \int_0^1 \frac{1}{(1-t)^2} \int_0^{1-t} (1-t) m(1-t) \dd r \dd t = \int_0^1 m(1-t) \dd t < \infty 
\]
which concludes the proof of \eqref{eq:barely:tocheck}.  

This concludes the proof of Theorem~\ref{t:intro5} modulo Lemma~\ref{lemma:diss}, which we prove next.
\end{proof}

\begin{proof}[Proof of Lemma~\ref{lemma:diss}]

The proof is similar to that of Lemma~\ref{lemma:dissipation}, but several difficulties arise because we lost the homogeneity of the kernel of $\EL$. 

Recall that $\mathop{supp} \phi_a \cap \{ x_2>0\} = B_{a/4}(0,a) =: \Omega_a$. 
We prove prove that on $\partial\Omega_a$, we have that $\EL \phi_a \leq - \delta_a < 0$ for some $\delta_a>0$. The reason this holds is that point on $\partial \Omega_a$ are local minima of $\phi_a$, and the positive contribution from the lower half plane is dominated by the negative contribution from the upper half plane. Let $x \in \partial\Omega_a$. By the definition of $\phi$ in \eqref{eq:phi:barely} we have that $\phi_a(x) = 0$ and therefore
\begin{align} 
\EL \phi_a(x) &= - \int_{\RR^{2}} \phi_a(y) \frac{m(|x-y|)}{|x-y|^{2}} \dd y \notag\\
&= - \int_{\Omega_a} \eta\left( \frac{4 |(y_1,y_2-a)|}{a}\right) \frac{m(|x-y|)}{|x-y|^{2}} \dd y + \int_{\Omega_a'} \eta\left( \frac{4 |(y_1,y_2+a)|}{a}\right) \frac{m(|x-y|)}{|x-y|^{2}} \dd y\notag\\
& = - \int_{\Omega_a} \eta\left( \frac{4 |(y_1,y_2-a)|}{a}\right) \left( \frac{m(|x-y|)}{|x-y|^{2}} - \frac{m(|x-Ry|)}{|x-Ry|^{2}} \right) \dd y
\label{eq:L:1}
\end{align}
where for $y = (y_1,y_2)$ we have denoted $Ry=(y_1,-y_2)$. For $x \in \partial \Omega_a$ and $y \in \Omega_a$, we have that $|x-y| \leq |x-Ry|$, and hence due to the monotonicity of $m$, the integrand in \eqref{eq:L:1} is positive. Coupled with the fact that $\eta \geq 0$, this already shows $\EL \phi_a(x) \leq 0$. Note that for $x \in \partial B_{a/4}(0,a)$ and $y \in B_{a/4}(0,a)$, by the triangle inequality we have $|x-Ry| \geq 2 |x-y|$. This can be seen by drawing a picture. Therefore, since $m$ is decreasing we obtain from \eqref{eq:L:1} that
\begin{align} 
\EL \phi_a(x) &\leq -  \int_{B_{a/8}(0,a)} \eta\left( \frac{4 |(y_1,y_2-a)|}{a}\right) \left( \frac{m(x-Ry)}{|x-y|^2} - \frac{m(x-Ry)}{4 |x-y|^2} \right) \dd y\notag\\
&\leq -  \frac 34 \int_{B_{a/8}(0,a)}  \frac{m(3 a)}{|x-y|^2}  \dd y \leq - \frac{m(3a)}{8} =: -\delta_a \label{eq:L:2}
\end{align}
and \eqref{eq:eigen} holds for $x \in \partial \Omega_a$.

Similarly, since $|x-y|\leq |x-Ry|$, whenever both $x$ and $y$ are in the upper half-space, on the set 
$\{x_2 >0 \} \cap \Omega_a^C$ 
we have $\EL \phi_a \leq 0$, so that \eqref{eq:eigen} trivially holds.

By smoothness of $\eta$, there exists $\rho_a  \in (0, a/4) $ such that on the annulus 
$B_{a/4}(0,a) \setminus B_{a/4 - \rho_a}(0,a)$
we have $ \EL \phi_a \leq 0 $. We have to estimate $\rho_a$ from below, as this will be needed later on. Let $x \in B_{a/4}(0,a) \setminus B_{a/4 - \rho_a}(0,a)$ for some $\rho_a>0$. From the mean value theorem, setting
\begin{align} 
\rho_a = \min \left\{ \frac{\delta_a}{\| \nabla \EL \phi_a \|_{L^\infty}}, \frac{a}{8} \right\} \label{eq:rho:1}
\end{align}
ensures that $\EL\phi_a(x) \leq - \delta_a/2$. In addition, we have
\begin{align} 
|\nabla \EL \phi_a(x)| &\leq \int_{\RR^2} \left| \nabla \phi_a(x) - \nabla\phi_a(x+y) \right| \frac{m(y)}{|y|^2} \dd y \notag \\
&\leq c \| \nabla \phi_a \|_{C^1} \int_{|y|\leq 1}  \frac{m(y)}{|y|} \dd y + c \|\nabla \phi_a\|_{L^\infty} \int_{|y|\geq 1}  \frac{m(y)}{|y|^2} \dd y\notag\\
&\leq \frac{c \| \eta \|_{C^2}}{a}  \int_0^1 m(r) \dd r + \frac{c \| \eta \|_{C^1}}{a}  \int_1^\infty \frac{m(r)}{r} \dd r =: \frac{c_{\eta,m}}{a} \label{eq:grad:L}
\end{align}
in view of \eqref{eq:m:cond:1}. It follows from \eqref{eq:rho:1} and \eqref{eq:grad:L} that
\begin{align} 
\frac{1}{2} \geq \frac{4 \rho_a}{a} \geq \frac{1}{2} \min \left\{ c_{\eta,m}^{-1} m(3a), 1\right\} \geq  \frac{1}{2} \min \left\{ c_{\eta,m}^{-1} m(3), 1\right\} =: c_3 \label{eq:rho:2}
\end{align}
since $m$ is monotone decreasing and $a\leq 1$. Here $c_3 \leq 1/2$ is independent of $a$.

Lastly, if $x \in B_{a/4- \rho_a}(0,a)$, then since $\eta$ is radially non-increasing, by \eqref{eq:rho:2} we have
\begin{align} 
\phi_a(x) = \eta\left( \frac{4|(x_1,x_2-a)|}{a}\right) \geq \eta \left(1 - \frac{4 \rho_a}{a}\right) \geq  c_\eta > 0
\end{align}
where $c_\eta$ is independent of $a$.
Therefore,  to ensure that $\EL \phi_a(x) \leq H(a) \phi_a(x)$, we just need to verify
\begin{align} 
H(a) \geq \frac{\| \EL \phi_a\|_{L^\infty} }{c_\eta}. \label{eq:gamma:lower}
\end{align}
Similarly to \eqref{eq:grad:L} we may bound (but this time we split the integral domains at $|y|=a$ not at $|y|=1$, and we exploit the P.V. in the definition of $\EL$ to write the nonlocal operator in terms of double-differences)
\begin{align} 
|\EL \phi_a(x)| &\leq \frac{1}{2} \int_{\RR^2} \left| 2 \phi_a(x) - \phi_a(x+y) - \phi_a(x-y)\right| \frac{m(y)}{|y|^2} \dd y \notag \\
&\leq \frac{c \| \eta\|_{C^2}}{a^2} \int_0^{a} r m(r) \dd r + c \int_{a}^1 \frac{m(r)}{r} \dd r + c \int_1^\infty \frac{m(r)}{r}\dd r \notag\\
&\leq c_0 \left( \frac{1}{a^2} \int_0^{a} r m(r) \dd r +  \int_{a}^1 \frac{m(r)}{r} \dd r +1 \right) \label{eq:sup:L}
\end{align}
for some $c_{0}>0$ independent of $a$. Combining \eqref{eq:gamma:lower}--\eqref{eq:sup:L} completes the proof of the lemma.
\end{proof}

\bibliographystyle{siam}

\end{document}